\documentclass[12pt]{article}

\usepackage{bm}
\usepackage{amsmath}
\usepackage{amsfonts}
\usepackage{color}
\usepackage{amssymb}
\usepackage{graphicx}
\usepackage{enumerate}
\usepackage{amsthm,amscd}
\usepackage[all]{xy}

\allowdisplaybreaks

\usepackage{hyperref}

\newtheorem{theorem}{Theorem}[section]
\newtheorem{construction}[theorem]{Construction}
\newtheorem{corollary}[theorem]{Corollary}
\newtheorem{definition}[theorem]{Definition}

\newtheorem{lemma}[theorem]{Lemma}
\newtheorem{fact}[theorem]{Fact}
\newtheorem{proposition}[theorem]{Proposition}

\newtheorem{remark}[theorem]{Remark}
\newtheorem{claim}[theorem]{Claim}

\newtheorem{obser}[theorem]{Observation}
\newtheorem{assumption}[theorem]{Assumption}

\begin{document}

\title{Mixed pairwise cross intersecting families (I)\thanks{This work is supported by  NSFC (Grant No. 11931002).  E-mail addresses: yangmiemie@hnu.cn (Yang Huang),
ypeng1@hnu.edu.cn (Yuejian Peng, corresponding author).}}

\author{Yang Huang, Yuejian Peng$^{\dag}$ \\[2ex]
{\small School of Mathematics, Hunan University} \\
{\small Changsha, Hunan, 410082, P.R. China }  }

\maketitle

\vspace{-0.5cm}

\begin{abstract}
Families $\mathcal{A}$ and $\mathcal{B}$ are cross-intersecting if $A\cap B\ne \emptyset$ for any $A\in \mathcal{A}$ and $B\in \mathcal{B}$. If $n<k+l$, all families $\mathcal{A}\subseteq {[n]\choose k}$ and $\mathcal{B}\subseteq {[n]\choose l}$ are cross intersecting  and we say that  $\mathcal{A}$ and $\mathcal{B}$ are cross intersecting freely. An $(n, k_1, \dots, k_t)$-cross intersecting system is  a set of non-empty pairwise cross-intersecting families $\mathcal{F}_1\subset{[n]\choose k_1}, \mathcal{F}_2\subset{[n]\choose k_2}, \dots, \mathcal{F}_t\subset{[n]\choose k_t}$  with $t\geq 2$ and $k_1\geq k_2\geq \cdots \geq k_t$. If an $(n, k_1, \dots, k_t)$-cross intersecting system contains at least two families which are cross intersecting freely and at least two families which are cross intersecting but not freely, then we say that the cross intersecting system is of mixed type. 
All previous studies are on non-mixed type, i.e, under the condition that  $n \ge k_1+k_2$. In this paper, we study for the first interesting mixed type, an $(n, k_1, \dots, k_t)$-cross intersecting system with $k_1+k_3\leq n <k_1+k_2$, i.e.,  families $\mathcal{F}_i\subseteq {[n]\choose k_i}$ and $\mathcal{F}_j\subseteq {[n]\choose k_j}$ are cross intersecting freely if and only if $\{i, j\}=\{1, 2\}$.  Let $M(n, k_1, \dots, k_t)$ denote the maximum sum of sizes of families in an $(n, k_1, \dots, k_t)$-cross intersecting system. 
We determine $M(n, k_1, \dots, k_t)$ and characterize all  extremal $(n, k_1, \dots, k_t)$-cross intersecting systems for $k_1+k_3\leq n <k_1+k_2$.

 A result of Kruskal-Katona  allows us to consider only families $\mathcal{F}_i$ whose elements are the first $|\mathcal{F}_i|$ elements in lexicographic order (we call them  L-initial families). Since $n <k_1+k_2$, $\mathcal{F}_1\subseteq {[n]\choose k_1}$ and $\mathcal{F}_2\subseteq {[n]\choose k_2}$ are cross intersecting freely. Thus, when we try to  bound $\sum_{i=1}^t{|\mathcal{F}_i|}$ by a function, there are two free variables $I_1$ (the last element of $\mathcal{F}_1$) and $I_2$ (the last element of $\mathcal{F}_2$). This causes more difficulty to analyze properties of the corresponding function comparing to non-mixed type problems (single-variable function). To overcome this difficulty, we introduce new concepts `$k$-partner' and `parity', develop some rules to determine whether two L-initial cross intersecting families are maximal to each other, and  prove one crucial property that in an extremal L-initial  $(n, k_1, \dots, k_t)$-cross intersecting system  ($\mathcal{F}_1$, $\mathcal{F}_2$, $\cdots$, $\mathcal{F}_t$), the last element of $\mathcal{F}_1$ and the last element of $\mathcal{F}_2$ are `parities' (see Section \ref{sec2} for the definition) to each other. This discovery allows us to bound $\sum_{i=1}^t{|\mathcal{F}_i|}$ by a single variable function $g(I_2)$, where $I_2$ is the last element of $\mathcal{F}_2$. Another crucial and challenge part is  to verify that $-g(I_2)$ has unimodality. Comparing to the non-mixed type, we need to overcome more difficulties in showing the unimodality of  function $-g(I_2)$ since there are more terms to be taken care of. 
   We think that the characterization of maximal cross intersecting  L-initial families and the unimodality of functions in this paper are interesting in their own, in addition to the extremal result.
   The most general condition on $n$  is that  $n\ge k_1+k_t$ (If $n< k_1+k_t$, then $\mathcal{F}_1$ is cross intersecting  freely with any other  $\mathcal{F}_i$, and  consequently  $\mathcal{F}_1={[n]\choose k_1}$ in an extremal $(n, k_1, \dots, k_t)$-cross intersecting system and we can remove $\mathcal{F}_1$ .). This paper provides foundation work for the solution to the most general condition $n\ge k_1+k_t$.

\end{abstract}



\section{Introduction}
Let $[n]=\{1, 2, \dots, n\}$.  For $0\leq k \leq n$, let ${[n]\choose k}$ denote the family of all $k$-subsets of $[n]$. A family $\mathcal{A}$ is $k$-uniform if $\mathcal{A}\subset {[n]\choose k}$. A family $\mathcal{A}$ is  {\em intersecting} if $A\cap B\ne \emptyset$ for any $A$ and $B\in \mathcal{A}$. Many researches in extremal set theory are inspired by the foundational  result of Erd\H{o}s--Ko--Rado \cite{EKR1961} showing that a maximum $k$-uniform intersecting family is a full star.  This result of Erd\H{o}s--Ko--Rado has many interesting generalizations. Two families $\mathcal{A}$ and $\mathcal{B}$ are {\em cross-intersecting} if $A\cap B\ne \emptyset$ for any $A\in \mathcal{A}$ and $B\in \mathcal{B}$. Note that $\mathcal{A}$ and $\mathcal{A}$ are  cross-intersecting is equivalent to that $\mathcal{A}$ is intersecting.
We call $t$ $(t\geq 2)$ families  $\mathcal{A}_1, \mathcal{A}_2,\dots, \mathcal{A}_t$ {\em pairwise cross-intersecting families} if $\mathcal{A}_i$ and $\mathcal{A}_j$ are cross-intersecting when $1\le i<j \le t$. Additionally, if $\mathcal{A}_j\ne \emptyset$ for each $j\in [t]$, then we say that $\mathcal{A}_1, \mathcal{A}_2,\dots, \mathcal{A}_t$ are {\em non-empty pairwise cross-intersecting}.
For given integers $n, t, k_t, \dots, k_t$, if families $\mathcal{A}_1\subseteq {[n]\choose k_1}, \dots, \mathcal{A}_t\subseteq {[n]\choose k_t}$ are non-empty pairwise cross intersecting and $k_1\geq k_2\dots\geq k_t$, then we call $(\mathcal{A}_1, \dots, \mathcal{A}_t)$  an {\it $(n, k_1, \dots, k_t)$-cross intersecting system}.

Define
\begin{align}\label{newdef1}
 M(n, k_1, \dots, k_t)&=\max \bigg\{\sum_{i=1}^t|\mathcal{A}_i|: (\mathcal{A}_1, \dots, \mathcal{A}_t) \text { is an } (n, k_1, \dots, k_t)\text{-cross} \nonumber \\
 & \quad\quad\quad\quad  \text{intersecting system }\bigg\}.
 \end{align}
We say that an $(n, k_1, \dots, k_t)$-cross intersecting system $(\mathcal{A}_1, \dots, \mathcal{A}_t)$  is {\em extremal} if $\sum_{i=1}^t|\mathcal{A}_i|=M(n, k_1, \dots, k_t)$. In \cite{HP} (Theorem \ref{HP}), the authors determined $M(n, k_1, \\ \dots, k_t)$ for $n\geq k_1+k_2$  and  characterized the extremal families attaining the bound.

\begin{theorem}[Huang-Peng \cite{HP}]\label{HP}
Let $\mathcal{A}_1\subset{[n]\choose k_1}, \mathcal{A}_2\subset{[n]\choose k_2}, \dots, \mathcal{A}_t\subset{[n]\choose k_t}$ be non-empty pairwise cross intersecting families with $t\geq 2$, $k_1\geq k_2\geq \cdots \geq k_t$, and $n\geq k_1+k_2$. Then
$$
\sum_{i=1}^t{|\mathcal{A}_i|}\leq \textup{max} \left\{{n\choose k_1}-{n-k_t\choose k_1}+\sum_{i=2}^t{{n-k_t\choose k_i-k_t}}, \,\,\sum_{i=1}^t{n-1\choose k_i-1}\right\}.
$$
The equality holds if and only if one of the following holds.\\
(i) ${n\choose k_1}-{n-k_t\choose k_1}+\sum_{i=2}^t{{n-k_t\choose k_i-k_t}}>\sum_{i=1}^t{n-1\choose k_i-1}$, and there is some $k_t$-element set $T\subset [n]$ such that $\mathcal{A}_1=\{F\in {[n]\choose k_1}: F\cap T\ne\emptyset\}$ and $\mathcal{A}_j=\{F\in {[n]\choose k_j}: T\subset F\}$ for each $j\in [2, t]$;\\
(ii)  ${n\choose k_1}-{n-k_t\choose k_1}+\sum_{i=2}^t{{n-k_t\choose k_i-k_t}}\le\sum_{i=1}^t{n-1\choose k_i-1}$, there are some $i\neq j$ such that $n>k_i+k_j$, and there is some $a\in [n]$ such that $\mathcal{A}_j=\{F\in{[n]\choose k_j}: a\in F\}$ for each $j\in [t]$;\\
(iii) $t=2, n=k_1+k_2$, $\mathcal{A}_1\subset {[n]\choose k_1}$ and $\mathcal{A}_2={[n]\choose k_2}\setminus \overline{\mathcal{A}_1}$;\\
(iv) $t\geq 3, k_1=k_2=\cdots=k_t=k, n=2k$, fix some $i\in [t]$, $\mathcal{A}_j=\mathcal{A}$ for all $j\in [t]\setminus \{i\}$, where $\mathcal{A}$ is an intersecting family with size ${n-1\choose k-1}$, and
$\mathcal{A}_i={[n]\choose k}\setminus \overline{\mathcal{A}}$.
\end{theorem}

Theorem \ref{HP} generalized the results of Hilton-Milner in \cite{HM1967}, Frankl-Tokushige in \cite{FT} and Shi-Qian-Frankl in \cite{SFQ2020}.
In Theorem \ref{HP}, taking $t=2$ and $k_1=k_2$, we obtain the result of Hilton-Milner in \cite{HM1967}. Taking $t=2$, we obtain the result of Frankl-Tokushige in \cite{FT}. Taking $k_1=k_2=\dots=k_t$, we obtain the result of Shi-Qian-Frankl in \cite{SFQ2020}.
Note that if $n<k+\ell$, any family $\mathcal{A}\subseteq {[n]\choose k}$ and any family $\mathcal{B}\subseteq {[n]\choose \ell}$ are cross intersecting. In this case, we say that
families $\mathcal{A}\subseteq {[n]\choose k}$ and $\mathcal{B}\subseteq {[n]\choose \ell}$ are  {\it cross intersecting  freely}. If an $(n, k_1, \dots, k_t)$-cross intersecting system contains at least two families which are cross intersecting freely and at least two families which are  cross intersecting but not freely, then we say that the cross intersecting system is of {\em mixed type}. Otherwise, we say that it is of {\em non-mixed type}.
All previous studies are of non-mixed type, i.e., $ n \ge k_1+k_2$.  It would be interesting  to study for mixed type.
The first interesting case is to study an $(n, k_1, \dots, k_t)$-cross intersecting system
$(\mathcal{F}_1, \dots, \mathcal{F}_t)$ 
 for  $k_1+k_3\leq n <k_1+k_2$.
In this case, $\mathcal{F}_i\subseteq {[n]\choose k_i}$ and $\mathcal{F}_j\subseteq {[n]\choose k_j}$ are cross intersecting freely if and only if $\{i, j\}=\{1, 2\}$. In this paper, we focus on this case,
we determine
$M(n, k_1, \dots, k_t)$ and characterize  all extremal $(n, k_1, \dots, k_t)$-cross intersecting systems.
Let us look at the following $(n, k_1, \dots, k_t)$-cross intersecting systems.

\begin{construction}\label{con1}
Let $k_1\geq k_2\geq\dots\geq k_t$ and $k_1+k_3\leq n <k_1+k_2$. For each $i\in [t]$,  we denote
$$\mathcal{G}_i=\left\{ G\in {[n]\choose k_i}: 1\in G \right\}.$$
\end{construction}

Note that $$\sum_{i=1}^t|\mathcal{G}_i|=\sum_{i=1}^t{n-1\choose k_i-1}=:\lambda_1.$$

\begin{construction}\label{con2}
Let $k_1\geq k_2\geq\dots\geq k_t$ and $k_1+k_3\leq n <k_1+k_2$.  For $i=1$ or $2$, let
$$\mathcal{H}_i=\left\{ H\in {[n]\choose k_i}: H\cap [k_t]\ne\emptyset \right\},$$
and for $i\in[3, t]$, let
$$\mathcal{H}_i=\left\{ H\in {[n]\choose k_i}: [k_t]\subseteq H\right\}.$$
\end{construction}

Note that $$\sum_{i=1}^t|\mathcal{H}_i|=\sum_{i=1}^2\left({n\choose k_i}-{n-k_t\choose k_i}\right)+\sum_{i=3}^t{{n-k_t\choose k_i-k_t}}=:\lambda_2.$$

Our main result  is that an extremal  $(n, k_1, \dots, k_t)$-cross intersecting system must be isomorphic to Construction \ref{con1} or Construction \ref{con2} if $k_1+k_3\leq n<k_1+k_2$.

\begin{theorem}\label{main1}
Let $\mathcal{F}_1\subset{[n]\choose k_1}, \mathcal{F}_2\subset{[n]\choose k_2}, \dots, \mathcal{F}_t\subset{[n]\choose k_t}$ be non-empty pairwise cross intersecting families with $t\geq 3$, $k_1\geq k_2\geq \cdots \geq k_t$ and $k_1+k_3\leq n<k_1+k_2$. Then
$$
\sum_{i=1}^t{|\mathcal{F}_i|}\leq \textup{max} \left\{\sum_{i=1}^t{n-1\choose k_i-1}, \,\,\sum_{i=1}^2\left({n\choose k_i}-{n-k_t\choose k_i}\right)+\sum_{i=3}^t{{n-k_t\choose k_i-k_t}}\right\}.
$$
The equality holds if and only if $(\mathcal{F}_1, \dots, \mathcal{F}_t)$ is isomorphic to $(\mathcal{G}_1, \dots, \mathcal{G}_t)$ in Construction \ref{con1} or $(\mathcal{H}_1, \dots, \mathcal{H}_t)$ in Construction \ref{con2}.
\end{theorem}

 In proving  Theorem \ref{HP} in \cite{HP}, the authors applied a result of Kruskal-Katona (Theorem \ref{kk})  allowing us to consider only families $\mathcal{F}_i$ whose elements are the first $|\mathcal{F}_i|$ elements in lexicographic order (we call them  L-initial families). We bounded $\sum_{i=1}^t{|\mathcal{F}_i|}$ by a function  $f(R)$ of the last element $R$ in the lexicographic order of an  L-initial family $\mathcal{F}_1$ ($R$ is called the ID of  $\mathcal{F}_1$ ),  and showed that $-f(R)$ has   unimodality.  To prove our main result (Theorem \ref{main1}) in this paper, by the result of Kruskal-Katona (Theorem \ref{kk}), we can  still consider only  pairwise cross-intersecting non-empty L-initial families $\mathcal{F}_1\subset{[n]\choose k_1}, \mathcal{F}_2\subset{[n]\choose k_2}, \dots, \mathcal{F}_t\subset{[n]\choose k_t}$. However, in this paper, the condition on $n$ is relaxed to $k_1+k_3\leq n <k_1+k_2$, so
 $\mathcal{F}_1\subseteq {[n]\choose k_1}$ and $\mathcal{F}_2\subseteq {[n]\choose k_2}$ are cross intersecting freely. When we try to  bound $\sum_{i=1}^t{|\mathcal{F}_i|}$ by a function, there are two free variables $I_1$ (the ID of $\mathcal{F}_1$) and $I_2$ (the ID of $\mathcal{F}_2$). This causes more difficulty to analyze properties of the corresponding function, comparing to the problem in \cite{HP}. To overcome this difficulty,  we introduce new concepts `$k$-partner' and `parity', develop some rules to determine whether  a pair of L-initial cross intersecting families are maximal to each other (see precise definition), and
  prove one crucial property that for an extremal L-initial  $(n, k_1, \dots, k_t)$-cross intersecting system  ($\mathcal{F}_1$, $\mathcal{F}_2$, $\cdots$, $\mathcal{F}_t$), the ID $I_1$ of $\mathcal{F}_1$ and the ID $I_2$ of $\mathcal{F}_2$ are `parities' to each other (Lemma \ref{l2.18}). This discovery allows us to bound $\sum_{i=1}^t{|\mathcal{F}_i|}$ by a single variable function $g(I_2)$. Another crucial and challenge part is  to verify the  unimodality of $-g(I_2)$ (Lemmas \ref{c2.27}, \ref{c2.28}  and \ref{c2.29}). Comparing to the function in  \cite{HP}, we need to overcome more difficulties in dealing with  function $-g(I_2)$ since there are more `mysterious' terms to be taken care of.
  We take advantage of some properties of function $f(R)$  obtained in \cite{HP} and come up with some new strategies in estimating the change $g(I'_2)-g(I_2)$ as the ID of $\mathcal{F}_2$ increases from $I_2$ to $I'_2$ (Sections \ref{sec6} and \ref{sec7}).

  The most general condition on $n$  is that  $n\ge k_1+k_t$. If $n< k_1+k_t$, then $\mathcal{F}_1$ is cross intersecting  freely with any other  $\mathcal{F}_i$, and  consequently  $\mathcal{F}_1={[n]\choose k_1}$ in an extremal $(n, k_1, \dots, k_t)$-cross intersecting system and $\mathcal{F}_1$ can be removed.  The most general case is more complex and we will deal with it in a forthcoming paper \cite {mix2}. The work in  this paper provides important foundation  for the solution to the most general condition $n\ge k_1+k_t$.

  To obtain the relationship between the ID of $\mathcal{F}_1$ and the ID of $\mathcal{F}_2$, we build some foundation work in Section \ref{sec2}, for example, we come up with new concepts `$k$-partners' and `parity', and give a necessary and sufficient condition on maximal cross intersecting  L-initial families.  In Section \ref{sec3}, we give the proof of Theorem \ref{main1} by assuming the truth of Lemmas \ref{l2.18}, \ref{c2.27}, \ref{c2.28}  and \ref{c2.29}.  In Section \ref{sec4}, we list some results obtained for non-mixed type  in \cite{HP} which we will apply. In Sections \ref{sec6} and \ref{sec7}, we give the proofs of Lemmas \ref{l2.18}, \ref{c2.27}, \ref{c2.28}  and \ref{c2.29}.




\section{Partner and Parity}\label{sec2}
In this section, we introduce new concepts `$k$-partner' and `parity', develop some rules to determine whether a pair of  L-initial cross intersecting families are maximal to each other (precise definitions will be given in this section). We prove some properties which are foundation for the proof of our main results.

When we write a set $A=\{a_1, a_2, \ldots, a_s\}\subset [n]$, we always assume that $a_1<a_2<\ldots<a_s$ throughout the paper. Let $\max A$ denote the maximum  element of $A$, let $\min A$ denote the minimum element of $A$ and  $(A)_i$ denote the $i$-th element of $A$. Let us introduce the {\it lexicographic (lex for short) order} of subsets of positive integers. Let $A$ and $B$ be finite subsets of the set of positive integers $\mathbb{Z}_{>0}$. We say that $A\prec B$ if either $A\supset B$ or $\min(A\setminus B) < \min(B\setminus A)$. In particular, $A\prec A$. Let $A$ and $B$ in ${[n]\choose k}$ with $A\precneqq B$. We write $A<B$ if there is no other   $C\in { [n]\choose k}$ such that $A\precneqq C\precneqq B$.

Let $\mathcal{L}(n, r, k)$ denote the first $r$ subsets in ${[n]\choose k}$ in the lex order. Given a set $R$, we denote $\mathcal{L}(n, R, k)=\{F\in {[n]\choose k}: F\prec R\}$. Whenever the underlying set $[n]$ is clear, we shall ignore it and write $\mathcal{L}(R, k)$, $\mathcal{L}(r, k)$ for short.  Let $\mathcal{F}\subset {[n]\choose k}$ be a family, we say $\mathcal{F}$ is {\it L-initial} if $\mathcal{F}=\mathcal{L}(R, k)$ for some $k$-set $R$.  We call $R$ the  {\it ID} of $\mathcal{F}$.

The well-known Kruskal-Katona theorem \cite{KK1, KK2} will allow us to consider only L-initial families. An equivalent formulation of this result was given in \cite{KK3, KK4} as follows.

\begin{theorem}[Kruskal-Katona theorem \cite{KK3, KK4}]\label{kk}
For $\mathcal{A}\subset {[n]\choose k}$ and $\mathcal{B}\subset {[n]\choose l}$, if $\mathcal{A}$ and $\mathcal{B}$ are cross intersecting, then $\mathcal{L}(|\mathcal{A}|, k)$ and $\mathcal{L}(|\mathcal{B}|, l)$ are cross intersecting as well.
\end{theorem}

In \cite{HP}, we proved the following important result.
\begin{proposition}\cite{HP}\label{p2.7}
Let $a, b, n$ are positive integers and $a+b\leq n$. For $P\subset [n]$ with $|P|\leq a$, let $Q$ be the partner of $P$. Then $\mathcal{L}(Q, b)$ is the maximum L-initial $b$-uniform family that are cross intersecting to $\mathcal{L}(P, a)$.
\end{proposition}

In \cite{HP},  we worked on non-mixed type: Let $t\geq 2$, $k_1\geq k_2\geq \cdots \geq k_t$ and $n\geq k_1+k_2$  and families $\mathcal{A}_1\subset{[n]\choose k_1}, \mathcal{A}_2\subset{[n]\choose k_2}, \dots, \mathcal{A}_t\subset{[n]\choose k_t}$ be non-empty pairwise cross-intersecting (not freely). Let $R$ be the ID of $\mathcal{A}_1$, and $T$  be the partner of $R$. In the proof of Theorem \ref{HP}, one important ingredient is that by Proposition \ref{p2.7}, $\sum_{i=1}^t{|\mathcal{A}_i|}$ can be bounded by a function of $R$ as following.

\begin{align}\label{new}
f(R)=\sum_{j=1}^t|\mathcal{A}_j|\leq |\mathcal{L}(R, k_1)| +\sum_{j=2}^t |\mathcal{L}(T, k_j)| .
\end{align}

By Theorem \ref{kk},  to prove the quantitative part of Theorem \ref{main1} we may  also assume that $\mathcal{F}_i$ is L-initial, that is, $\mathcal{F}_i=\mathcal{L}(|\mathcal{F}_i|, k_i)$ for each $i\in [t]$. In the rest of this chapter, {\bf we assume that $(\mathcal{F}_1, \dots, \mathcal{F}_t)$ is an extremal non-empty $(n, k_1, \dots, k_t)$-cross intersecting system with $t\geq 3$, $k_1+k_3\leq n< k_1+k_2$, and $\mathcal{F}_j$ is L-initial for each $j\in [t]$ with {\it ID} $I_j$.} 

However,  the condition on $n$ is relaxed to $k_1+k_3\leq n <k_1+k_2$, so
 $\mathcal{F}_1\subseteq {[n]\choose k_1}$ and $\mathcal{F}_2\subseteq {[n]\choose k_2}$ are cross intersecting freely. When we try to  bound $\sum_{i=1}^t{|\mathcal{F}_i|}$ by a function, there are two free variables $I_1$ (the ID of $\mathcal{F}_1$) and $I_2$ (the ID of $\mathcal{F}_2$). This causes more difficulty to analyze properties of the corresponding function, comparing to the problem in \cite{HP}. To overcome this difficulty,  we introduce new concepts '$k$-partner' and 'parity', develop some rules to determine whether  a pair of L-initial cross intersecting families are maximal to each other (see precise definition).

Let $\mathcal{F}\subseteq {[n]\choose f}$ and  $\mathcal{G}\subseteq {[n]\choose g}$ be cross intersecting families. We say that $(\mathcal{F}, \mathcal{G})$ is {\it maximal} or $\mathcal{F}$ and $\mathcal{G}$ are {\it maximal cross intersecting families}  if whenever $\mathcal{F}'\subseteq {[n]\choose f}$ and  $\mathcal{G}'\subseteq {[n]\choose g}$ are cross intersecting with $\mathcal{F}\subseteq \mathcal{F}'$ and $\mathcal{G}\subseteq\mathcal{G}'$, then $\mathcal{F}=\mathcal{F}'$ and $\mathcal{G}=\mathcal{G}'$.

Let $F$ and $G$ be two subsets of $[n]$. We  say $(F, G)$ is {\it maximal} if there are two L-initial families $\mathcal{F}\subseteq{[n]\choose |F|}$ and $\mathcal{G}\subseteq{[n]\choose |G|}$ with IDs $F$ and $G$ respectively such that $(\mathcal{F}, \mathcal{G})$ is maximal.
We say two families $\mathcal{A}_1$ and $\mathcal{A}_2$ are \emph{maximal pair families} if $|\mathcal{A}_1|=|\mathcal{A}_2|$ and for every $A_1\in \mathcal{A}_1$, there is a unique $A_2\in \mathcal{A}_2$ such that $(A_1, A_2)$ is maximal.

Let $F=\{x_1, x_2, \dots, x_k\}\subseteq [n]$. We denote
$$\ell(F)=
\begin{cases}
\max \{x: [n-x+1, n]\subseteq F\}, & \text{if $\max F=n$};\\
0, & \text{if $\max F<n$}.
\end{cases}
$$

Let $F\subseteq [n]$ be a set. We denote
$$F^{\mathrm{t}}=
\begin{cases}
\emptyset, & \text{if $\ell(F)=0$};\\
[n-\ell(F)+1, n], & \text{if $\ell(F)\geq 1$}.
\end{cases}
$$

Let $F$ and $H$ be two subsets of $[n]$ with size $|F|=f$ and $|H|=h$. We say that $F$ and $H$ {\it strongly intersect} at their last element if there is an element $q$ such that $F\cap H=\{q\}$ and $F\cup H=[q]$. We also say $F$ is the {\it partner} of $H$, or $H$ is the {\it partner} of $F$. Let $k\leq n-f$ be an integer, we define the {\it $k$-partner} $K$ of $F$ as follows. For $k=h$, let $K=H$.  If $k>h$, then let $K= H\cup \{n-k+h+1, \dots, n\}$. We can see that $|K|=k$. Indeed, since $F$ and $H$ intersect at their last element, $n'=\max H=f+h-1<n-k+h+1$, so $|K|=| H\cup \{n-k+h+1, \dots, n\}|=k$.  If $k< h$, then let $K$ be the last $k$-set in ${[n]\choose k}$ such that $K \prec H$, in other words, there is no $k$-set $K'$ satisfying $K\precneqq K'\precneqq H$. By the definition of $k$-partner, we have the following remark.

\begin{remark}\label{r2.6}
Let $F\subseteq [n]$ with $|F|=f$ and $k\leq n-f$. Suppose that  $H$ is the partner of $F$, and $K$ is the k-partner of $F$, then we have $\mathcal{L}(H, k)=\mathcal{L}(K, k)$.
\end{remark}

\begin{fact}\label{fact2.4+}
Let $F\subseteq [n]$ with $|F|=f$ and $k\leq n-f$. Then the $k$-partner of $F$ is the same as the $k$-partner of $F\setminus F^{\mathrm{t}}$.
\end{fact}

\begin{proof}
If $\ell(F)=0$, then we are fine. Suppose $\ell(F)>0$ and $F=\{x_1, \dots, x_y\}\cup\{n-\ell(F)+1, \dots, n\}$.
Let $H$ and $H'$ be the partners of $F$ and $F\setminus \{n-\ell(F)+1, \dots, n\}$ respectively. Then $|H|>n-f$, consequently $k<|H|$, 
$H=H\cap [x_y-1]\cup [x_y+1, n-\ell(F)]\cup \{n\}$ and
$H'=H\cap [x_y-1]\cup \{x_y\}$.
Suppose that $K$ and $K'$ are the $k$-partners of $F$ and $F\setminus \{n-\ell(F)+1, \dots, n\}$ respectively. By the definition of $k$-partner, we can see that
if $k\leq |H\cap [x_y-1]|$,  then $K=K'$, as desired; if $k= |H'|$, then $K'=H'$ and $K=H\cap [x_y-1]\cup \{x_y\}=H'$, as desired; if $k>|H'|$, then $K'=H'\cup [n-k+|H'|+1, n]$ and $K=H\cap [x_y-1]\cup \{x_y\}\cup [n-k+|H'|+1, n]=K'$, as desired.
\end{proof}

By Remark \ref{r2.6} and Proposition \ref{p2.7}, we have the following fact.
\begin{fact}\label{f2.8}
Let $a, b, n$ be positive integers and $n\geq a+b$. For  $A\subset [n]$ with $|A|=a$, let $B$ be the $b$-partner of $A$, then $\mathcal{L}(B, b)$ is the maximum L-initial $b$-uniform family that are cross intersecting to $\mathcal{L}(A, a)$. We also say that $\mathcal{L}(B, b)$ is maximal to $\mathcal{L}(A, a)$, or say $B$ is maximal to $A$.
\end{fact}

\begin{remark}\label{r2.9}
Note that  families $\mathcal{L}(A, a)$ and $\mathcal{L}(B, b)$ which mentioned in Fact \ref{f2.8} may not be maximal cross intersecting, since we don't know whether $\mathcal{L}(A, a)$ is maximal to $\mathcal{L}(B, b)$. For example, let $n=9$, $a=3$, $b=4$ and $A=\{2, 4, 7\}$. Then the $b$-partner of $A$ is $\{1, 3, 4, 9\}$. Although $\mathcal{L}(\{1, 3, 4, 9\}, 4)$ is maximal to $\mathcal{L}(\{2, 4, 7\}, 3)$, $\mathcal{L}(\{1, 3, 4, 9\}, 4)$ and $\mathcal{L}(\{2, 4, 7\}, 3)$ are not maximal cross intersecting families since $\mathcal{L}(\{2, 4, 7\}, 3)\subsetneq\mathcal{L}(\{2, 4, 9\}, 3)$, and $\mathcal{L}(\{2, 4, 9\}, 3)$ and $\mathcal{L}(\{1, 3, 4, 9\}, 4)$ are cross intersecting families.
\end{remark}


Frankl-Kupavskii \cite{FK2018} gave a sufficient condition for a pair of maximal cross  intersecting families, and a necessary condition for a pair of maximal cross  intersecting families as stated below.
\begin{proposition}[Frankl-Kupavskii \cite{FK2018}]\label{FK+}
Let $a, b\in \mathbb{Z}_{>0}, a+b\leq n$. Let $P$ and $Q$ be non-empty subsets of $[n]$ with $|P|\leq a$ and $|Q|\leq b$. If $Q$ is the partner of $P$, then $\mathcal{L}(P, a)$ and $\mathcal{L}(Q, b)$ are maximal cross intersecting families.
Inversely, if $\mathcal{L}(A, a)$ and $\mathcal{L}(B, a)$ are maximal cross intersecting families, let $j$ be the smallest element of $A\cap B$, $P=A\cap [j]$ and $Q=B\cap [j]$. Then $\mathcal{L}(P, a)=\mathcal{L}(A, a)$, $\mathcal{L}(Q, b)=\mathcal{L}(B, b)$ and $P$, $Q$ satisfy the following conditions:  $|P|\leq a$, $|Q|\leq b$, and $Q$ is the partner of $P$.
\end{proposition}

Based on Proposition \ref{FK+}, we point out a necessary and sufficient condition for a pair of maximal cross intersecting families in terms of their IDs. 

\begin{fact}\label{fact2.5}
Let $A$ and $B$ be nonempty subsets of $[n]$ with $|A|+|B|\leq n$. Let $A'=A\setminus A^{\mathrm{t}}$ and $B'=B\setminus B^{\mathrm{t}}$. Then $(A, B)$ is maximal if and only if $A'$ is the partner of $B'$.
\end{fact}

\begin{proof}
Let $|A|=a$ and $|B|=b$.
From the definitions of $A'$ and $B'$, we can see that $\mathcal{L}(A, a)=\mathcal{L}(A', a)$ and $\mathcal{L}(B, b)=\mathcal{L}(B', b)$.
First we show the sufficiency.  Suppose that $A'$ is the partner of $B'$. Since $|A'|\leq |A|$ and $|B'|\leq |B|$, by Proposition \ref{FK+}, $\mathcal{L}(A', a)$ and $\mathcal{L}(B', b)$ are maximal cross intersecting families. Thus $\mathcal{L}(A, a)$ and $\mathcal{L}(B, b)$ are maximal cross intersecting families, in other words, $(A, B)$ is maximal.
Next, we show the necessity. Suppose that $(A, B)$ is maximal. Let $j$ be the smallest element of $A\cap B$, $P=A\cap [j]$ and $Q=B\cap [j]$. By Proposition \ref{FK+}, 
$\mathcal{L}(A, a)=\mathcal{L}(P, a)$, $\mathcal{L}(B, b)=\mathcal{L}(Q, b)$;
$|P|\leq a$, $|Q|\leq b$ and
 $P$ is the partner of $Q$.
Since $\mathcal{L}(A, a)=\mathcal{L}(P, a)$ and $P\subseteq A$, then we $A=P\cup \{n-a+|P|+1, \dots, n\}$. Similarly,  $B=Q\cup \{n-b+|Q|+1, \dots, n\}$.
By the definitions of $A'$ and $B'$, we have $A'=P$ and $B'=Q$. Since  $P$ is the partner of $Q$, $A'$ is the partner of $B'$.
\end{proof}

By Fact \ref{fact2.5} and the definition of the $k$-partner, we have the following property.
\begin{fact}\label{fact2.10}
Let $a, b, k, n$ be integers with  $n\geq \max\{a+b, a+k\}$. Let $A$ be an $a$-subset of $[n]$.  Suppose that $K$ is the $k$-partner of $A\setminus A^{\mathrm{t}}$ and there exists a $b$-set $B$  such that $(A, B)$ is maximal and let $b'=|B\setminus B^{\mathrm{t}}|$. Then $(A, K)$ is maximal if and only if $k\geq b'$.
\end{fact}



\begin{fact}\label{f2.10}
Let $a, b, n$ be positive integers and $n\geq a+b$. Suppose that $A$ is an $a$-subset of $[n]$, and $B$ is the $b$-partner of $A$. Let $A'$ be the  $a$-partner of $B$, then $(A', B)$ is maximal. Moreover, if $A'\ne A$, then $A\precneqq A'$.
\end{fact}

\begin{proof}
If $(A, B)$ is maximal, then $A$ is the $a$-partner of $B$ and $A'=A$, we are fine. Suppose that $(A, B)$ is not maximal.
By Fact \ref{f2.8},  $B$ is maximal to $A$, so $A$ is not maximal to $B$. Let $A'$ be the $a$-partner of $B$. By Fact \ref{f2.8} again,  $A'$ is maximal to $B$ and $A\precneqq A'$. Since $B$ is maximal to $A$, for any $b$-set $B'$ satisfying  $B\precneqq B'$, we have  $\mathcal{L}(B', b)$ and $\mathcal{L}(A', a)$ are not cross intersecting. So $B$ is maximal to $A'$. Hence $(A', B)$ is maximal.
\end{proof}

\begin{definition}\label{def+}
Let $h_1\leq h_2$ be positive integers, $H_1$ and $H_2$ be subsets of $[n]$ with sizes $h_1$ and $h_2$ respectively. We say $H_1$ is the {\it $h_1$-parity} of $H_2$, or $H_2$ is the $h_2$-parity of $H_1$ if $H_1\setminus H_1^{\rm t}=H_2\setminus H_2^{\rm t}$ and $\ell(H_2)-\ell(H_1)=h_2-h_1$.
\end{definition}

\begin{remark}
Let $d\leq f\leq h$ be positive integers and $F\subseteq [n]$ with $|F|=f$. Then $F$ has an $h$-parity if and only if $h-f\leq n-\ell(F)-\max(F\setminus F^{\rm t})-1$, and 
$F$ has an $d$-parity if and only if $d\geq f-\ell(F)$.
\end{remark}

Note that  for a given integer $k$ and a subset $A\subseteq [n]$, if $A$ has a $k$-parity, then it has the unique one. The following fact is derived from the above definition directly.
\begin{fact}\label{fact2.13+}
Let $h_1\leq h_2\leq h_3$ and $H_i$ be an $h_i$-set for $i\in [3]$. If $H_1$ is the $h_1$-parity of $H_2$ and $H_2$ is the $h_2$-parity of $H_3$, then $H_1$ is the $h_1$-parity of $H_3$. Also, if $H_3$ is the $h_3$-parity of $H_1$ and $H_2$ is the $h_2$-parity of $H_1$, then $H_3$ is the $h_3$-parity of $H_2$.
\end{fact}

\begin{fact}\label{fact2.7}
Let $a, b, k, n$ be positive integers and $n\geq \max\{a+k, b+k\}$. Let $A$  and $B$ be two subsets of $[n]$ with sizes $|A|=a$ and $|B|=b$. Suppose that $K_a$ and $K_b$ are the $k$-partners of $A$ and $B$ respectively. If $A\prec B$, then $K_b\prec K_a$. In particular, if $A$ is the $a$-parity of $B$, then $K_b= K_a$.
\end{fact}

\begin{definition}\label{def2.7}
Let $h_1\leq h_2$. For two families $\mathcal{H}_1\subseteq {[n]\choose h_1}$ and $\mathcal{H}_2\subseteq {[n]\choose h_2}$, we say that $\mathcal{H}_1$ is the $h_1$-parity of $\mathcal{H}_2$, or $\mathcal{H}_2$ is the $h_2$-parity of $\mathcal{H}_1$ if

(i) for any $H_1\in \mathcal{H}_1$, the $h_2$-parity of $H_1$ exists and must be in $\mathcal{H}_2$;

(ii) for any $H_2\in \mathcal{H}_2$, either $H_2$ has no $h_1$-parity or it's $h_1$-parity is in $\mathcal{H}_1$.
\end{definition}

\begin{proposition}\label{prop2.9}
Let $f, g, h, n$ be positive integers with $f\geq g$ and $n\geq f+h$. Let
\begin{align*}
&\mathcal{F}=\left\{ F\in{[n]\choose f}: \text{ there exists } H\in {[n]\choose h} \text{ such that } (F, H) \text{ is maximal }  \right\},\\
&\mathcal{G}=\left\{ G\in{[n]\choose g}: \text{ there exists } H\in {[n]\choose h} \text{ such that } (G, H) \text{ is maximal }  \right\}.
\end{align*}
Let $\mathcal{H}_{\mathcal{F}}\subseteq {[n]\choose h}$ and $\mathcal{H}_{\mathcal{G}}\subseteq {[n]\choose h}$ be the families
such that $\mathcal{F}$ and $\mathcal{H}_{\mathcal{F}}$ are maximal pair families and $\mathcal{G}$ and $\mathcal{H}_{\mathcal{G}}$ are  maximal pair families.
Then $\mathcal{F}$ is the $f$-parity of $\mathcal{G}$; $\mathcal{H}_{\mathcal{G}}\subseteq \mathcal{H}_{\mathcal{F}}$; and for any $G\in \mathcal{G}$, let $F\in \mathcal{F}$ be the $f$-parity of $G$ and $H\in \mathcal{H}$ such that $(F, H)$ is maximal, then $(G, H)$ is maximal.
\end{proposition}

\begin{proof}
If $f=g$, then $\mathcal{F}=\mathcal{G}$ and $\mathcal{H}_{\mathcal{G}}=\mathcal{H}_{\mathcal{F}}$. So we may assume that $f=g+s$ for some $s\geq 1$. For any $G\in \mathcal{G}$, there is the unique $H\in \mathcal{H}$ such that $(G, H)$ is maximal. Let $G'=G\setminus G^{\rm t}=G\setminus \{n-\ell(G)+1, \dots, n\}$ and $H'=H\setminus H^{\rm t}=H\setminus \{n-\ell(H)+1, \dots, n\}$. Then, by Fact \ref{fact2.5}, $G'$ and $H'$ are partners of each other . So $\max G'\leq g+h-1\leq f+h-s-1$. Let $F=G'\cup \{n-\ell(G)+1-s, \dots, n\}$. Then $G\subseteq F$, $|F|=f$ and $\ell(F)-\ell(G)=f-g=s$. Then, by Definition \ref{def+}, $F$ is the $f$-parity of $G$.  Let $F'=F\setminus F^{\rm t}$. So $F'=G'$. Furthermore, $F'$ and $H'$ are partners of each other. By Fact \ref{fact2.5} again, we can see that $(F, H)$ is maximal. So $F\in \mathcal{F}$, and $\mathcal{F}$ is the $f$-parity of $\mathcal{G}$, as desired.
For any $G\in \mathcal{G}$,  let $F$ be the $f$-parity of $G$ and $H\in \mathcal{H}$ be the set such that $(F, H)$ is maximal.
By Fact \ref{fact2.7},  $(G, H)$ is maximal, as desired. And this implies $\mathcal{H}_{\mathcal{G}}\subseteq \mathcal{H}_{\mathcal{F}}$. 
The proof is complete.
\end{proof}

\begin{fact}\label{f2.11}
Let $a, b, c, n$ be positive integers, $a\geq b$ and $n\geq a+c$, and let $C$ be a $c$-subset of $[n]$. Suppose that $A$ is the $a$-partner of $C$ and $B$ is the $b$-partner of $C$, then $B\prec A$ or $A$ is the $a$-parity of $B$.
\end{fact}

\begin{proof}
If $a=b$, then $A=B$, we are done. Assume $a>b$.
Let $T$ be the partner of $C$ and let $|T|=c'$. If $c'\leq b$, then $c'<a$.  By the definitions of $a$-partner and $b$-partner and $n\geq a+c> b+c$, we can see that $A$ is the $a$-parity of $B$. If $b<c'\leq a$, then we have $\min B\setminus A < \min A\setminus B$, so $B\prec A$. At last,  suppose $b<a<c'$. Let $C'=\{x_1, \dots, x_b, \dots,  x_a\}\subseteq C$ be the first $a$ elements of $C$. If $x_{i+1}=x_i+1$ for all $i\in[b, a-1]$, then $A$ is the $a$-parity of $B$. Otherwise, we have $B\prec A$.
\end{proof}

\section{Proofs of Theorem \ref{main1}}\label{sec3}
{\bf Recall that $k_1+k_3\leq n <k_1+k_2$,  $(\mathcal{F}_1, \dots, \mathcal{F}_t)$ is an  extremal $(n, k_1, \dots, k_t)$-cross intersecting system, each $\mathcal{F}_i$ is L-initial, and $I_1, I_2, \dots, I_t$ are the IDs of $\mathcal{F}_1, \mathcal{F}_2, \dots, \mathcal{F}_t$ respectively  throughout the paper.}
From Constructions \ref{con1} and \ref{con2},  we have
\begin{equation}\label{+1}
\sum_{i=1}^t|\mathcal{F}_i|\geq \textup{max}\{\lambda_1, \,\, \lambda_2\}.
\end{equation}
 We are going to prove that $\sum_{i=1}^t|\mathcal{F}_i|\leq \textup{max}\{\lambda_1, \lambda_2\}$.

 \begin{claim}
 We may assume that $k_t\geq 2$.
 \end{claim}

 \begin{proof}
We  consider the case $k_t=1$. Suppose that $|\mathcal{F}_t|=s$. Then $\mathcal{F}_t=\{\{1\}, \dots, \{s\}\}$. Since $(\mathcal{F}_1, \dots, \mathcal{F}_t)$ is a cross intersecting system, then  for any $i\in [t-1]$ and $F\in \mathcal{F}_i$, we have $[s]\subseteq F$. Thus,
$$\sum_{i=1}^t|\mathcal{F}_i|=\sum_{i=1}^{t-1}{n-s\choose k_i-s}+s\leq\lambda_1.$$
So we may assume that $k_t\geq 2$.
\end{proof}

 From now on, $k_t\geq 2$.

 For the case $k_1=k_2$, the authors in \cite{HP+} (Corollary 1.12)  have already given a positive answer. Although we can prove for this case in this paper, for simplicity,  we  assume that
 \[
 k_1>k_2.
 \]

In \cite{HP+}, the authors proved the following result. 
\begin{theorem}\cite{HP+}\label{thm3.3}\label{theorem3.2}
Let $n$, $t\geq 2$, $k_1, k_2, \dots, k_t$ be positive integers and $d_1, d_2, \dots, d_t$ be positive numbers. Let
$\mathcal{A}_1\subset{[n]\choose k_1}, \mathcal{A}_2\subset{[n]\choose k_2}, \dots, \mathcal{A}_t\subset{[n]\choose k_t}$ be non-empty  cross-intersecting families with $|\mathcal{A}_i|\geq {n-1\choose k_i-1}$ for some $i\in [t]$. Let $m_i$ be the minimum integer among $k_j$, where $j\in [t]\setminus \{i\}$. If $n\geq k_i+k_j$ for all $j\in [t]\setminus \{i\}$, then
$$\sum_{1=j}^td_j|\mathcal{A}_j|\leq \max  \left\{d_i{n\choose k_i}-d_i{n-m_i\choose k_i}+\sum_{j=1, j\ne i}^td_j{n-m_i\choose k_j-m_i}, \,\,\sum_{j=1}^td_j{n-1\choose k_j-1}\right\}.$$
The equality holds if and only if one of the following holds.\\
(1) If $d_i{n\choose k_i}-d_i{n-k_t\choose k_i}+\sum_{j\ne i}^td_j{n-m_i\choose k_j-m_i}\geq \sum_{j=1}^td_j{n-1\choose k_j-1}$, then  there is some $m_i$-element set $T\subset [n]$ such that $\mathcal{A}_i=\{F\in {[n]\choose k_i}: F\cap T\ne\emptyset\}$ and $\mathcal{A}_j=\{F\in {[n]\choose k_j}: T\subset F\}$ for each $j\in [t]\setminus \{i\}$.\\
(2) If
$d_i{n\choose k_i}-d_i{n-k_t\choose k_i}+\sum_{j\ne i}^td_j{n-m_i\choose k_j-m_i}\leq \sum_{j=1}^td_j{n-1\choose k_j-1}$,
then there is some $a\in [n]$ such that $\mathcal{A}_j=\{F\in{[n]\choose k_j}: a\in F\}$ for each $j\in [t]$. \\
(3) If $t=2$ and $n=k_i+k_{3-i}$.  If $d_i\leq d_{3-i}$,  then $\mathcal{A}_{3-i}\subseteq {[n]\choose k_{3-i}}$ with $|\mathcal{A}_{3-i}|={n-1\choose k_{3-i}-1}$ and $\mathcal{A}_i={[n]\choose k_i}\setminus \overline{\mathcal{A}_{3-i}}$. \\
(4)
If $n=k_i+k_j$ holds for every $j\in [t]\setminus \{i\}$ and $\sum_{j\ne i}d_j=d_i$, then  $\mathcal{A}_j=\mathcal{A}$ for all $j\in [t]\setminus \{i\}$, where $\mathcal{A}\subseteq {[n]\choose k}$ is an intersecting family with size $|\mathcal{A}|={n-1\choose k-1}$, and $\mathcal{A}_i={[n]\choose k_i}\setminus \overline{\mathcal{A}}$. 
\end{theorem}

\begin{claim}\label{>}
 $|\mathcal{F}_1|\geq {n-1\choose k_1-1}$ and $|\mathcal{F}_2|\geq {n-1\choose k_2-1}$, in other words, $\{1, n-k_1+2, \dots, n\} \prec I_1$ and $\{1, n-k_2+2, \dots, n\} \prec I_2$.
\end{claim}

\begin{proof}
  If  $|\mathcal{F}_i|< {n-1\choose k_i-1}$ for each $i\in [t]$, then $\sum_{i=1}^t{|\mathcal{F}_i|}< \lambda_1$, a contradiction to (\ref{+1}). So there is $i\in [t]$ such that $|\mathcal{F}_i|\geq {n-1\choose k_i-1}$.

Suppose that $i\in [3, t]$ firstly. Since $ k_1\geq k_2$, $n\geq k_1+k_3$ and therefore $n\geq k_i+k_j$ for all $j\in [t]\setminus \{i\}$. Let $m_i=\min \{k_j, j\in [t]\setminus \{i\}\}$. 
Taking $d_1=d_2=\dots=d_t$ in Theorem \ref{thm3.3}, we obtain
\[
\sum_{j=1}^t|\mathcal{F}_j|
 \leq \textup{max}\left\{ \sum_{j=1}^t{n-1\choose k_j-1}, \,\,{n\choose k_i}-{n-k_t\choose k_i}+\sum_{j=1,j\ne i}^t{n-m_i\choose k_j-m_i}  \right\}.
 \]
In \cite{HP+} (Proposition 2.20 \cite{HP+}), we have shown that 
\[
{n\choose k_i}-{n-k_t\choose k_i}+\sum_{j=1,j\ne i}^t{n-m_i\choose k_j-m_i}
\leq {n\choose k_1}-{n-k_t\choose k_1}+\sum_{j=2}^t{n-k_t\choose k_j-k_t}.
\]
So
\begin{align*}
\sum_{j=1}^t|\mathcal{F}_j|&\leq \textup{max}\left\{ \sum_{j=1}^t{n-1\choose k_j-1},{n\choose k_1}-{n-k_t\choose k_1}+\sum_{j=2}^t{n-k_t\choose k_j-k_t}  \right\}  \\
&\leq\textup{max}\{\lambda_1, \,\, \lambda_2\},
\end{align*}
 where the last inequality holds by $k_t\geq 2$ and ${n\choose k_2}-{n-k_t\choose k_2}>{n-k_t\choose k_2-k_t}$. Thus if $\max \{\lambda_1, \,\, \lambda_2\}=\lambda_2$, then the last inequality holds strictly, it makes a a contradiction to (\ref{+1}). 
Suppose that $\max \{\lambda_1, \,\, \lambda_2\}=\lambda_2$.  
Then Claim \ref{>} follows from Theorem \ref{theorem3.2} (see the item (2) of Theorem \ref{theorem3.2}).
 So $i\in [2]$.
 Without loss of generality, let $i=1$. Since $n<k_1+k_2$, then  any two families  $\mathcal{G}\subseteq {[n]\choose k_1}$ and  $\mathcal{H}\subseteq{[n]\choose k_2}$ are cross intersecting freely. Since $|\mathcal{F}_1|\geq {n-1\choose k_1-1}$ and $\mathcal{F}_1$ is L-initial, $\{1, n-k_1+2, \dots, n\}\prec I_1$. Note that $n> k_1+k_j$ for each $j\in [3, t]$ and $\mathcal{F}_1$ is cross intersecting with $\mathcal{F}_j$ but not freely for each $j\in [3, t]$,  every member of $\mathcal{F}_j$ contains 1. Since $(\mathcal{F}_1, \dots, \mathcal{F}_t)$ is extremal, 
 all $k_2$-subsets contianing 1 are contained in $\mathcal{F}_2$, so $\{1, n-k_2+2, \dots, n\}\prec I_2$, this implies $|\mathcal{F}_2|\geq {n-1\choose k_2-1}$. This completes the proof of Claim \ref{>}.
\end{proof}

The following observation is simple and frequently used in this paper.
\begin{remark}\label{remark2.20+}
Let $d\geq 2$ be an integer, we consider $d$ L-initial families $\mathcal{L}(A_1, a_1)$, $\dots$, $\mathcal{L}(A_d, a_d)$. For $i\in [d]$ and let $S\subseteq [d]\setminus \{i\}$ be the set of all $j\in [d]\setminus \{i\}$ satisfying that $n\geq a_i+a_j$. Suppose that $\{1, n-a_i+2, \dots, n\}\preceq  A_i$. If $\mathcal{L}(A_i, a_i)$ and $\mathcal{L}(A_j, a_j)$ are cross intersecting for each $j\in S$, then
$\mathcal{L}(A_j, a_j)$, $j\in S$ are pairwise cross intersecting families since for any $j\in S$, every member of $\mathcal{L}(A_j, a_j)$ contains $1$.
\end{remark}

Combining Claim \ref{>}, Proposition \ref{p2.7} and $(\mathcal{F}_1, \dots, \mathcal{F}_t)$ is extremal, we have that for $i\in [3, t]$, $I_i$ is $k_i$-partner of $I_1$ (if $I_2\prec I_1$) or $I_2$ (if $I_1\prec I_2$).

Since $(\mathcal{F}_1, \dots, \mathcal{F}_t)$ is extremal, $I_i$ (ID of $\mathcal{F}_i$) must be contained in  $\mathcal{R}_i$, which are defined as below.
\begin{align*}
\mathcal{R}_1=\{R\in {[n]\choose k_1}:  \{1, n-k_1+2, \dots, n\} \prec R \prec \{k_t, n-k_1+2, \dots, n\}\}.\\
\mathcal{R}_2=\{ R\in {[n]\choose k_2}:  \{1, n-k_2+2, \dots, n\} \prec R \prec  \{k_t, n-k_2+2, \dots, n\}\}.
\end{align*}
 For $i\in[3, t-1]$, let
 \begin{align*}
\mathcal{R}_i&=\{ R\in {[n]\choose k_i}: [k_t]\cup[ n-k_i+k_t+1, n] \prec R \prec \{1, n-k_i+2, \dots, n\}\},\\
\mathcal{R}_t&=\{ R\in {[n]\choose k_t}:  \{1, 2, \dots, k_t\} \prec R \prec \{1, n-k_t+2, \dots, n\} \}.
\end{align*}

 For a family $\mathcal{A}_1\subseteq {[n]\choose k_1}$ with ID $A_1$, let
 \begin{align*}
m(n, A_{1})=&\max \big\{\sum_{j\in [t]\setminus \{1\}}|\mathcal{A}_j|: \mathcal{A}_j\subseteq {[n]\choose k_j} \text { is L-initial and } (\mathcal{A}_1, \dots, \mathcal{A}_t) \text{ is an }\\
& (n, k_1, \dots, k_t)\text{-cross intersecting system, where } k_1+k_3\leq n<k_1+k_2\big\}.
\end{align*}

 For L-initial cross intersecting families $\mathcal{A}_1\subseteq {[n]\choose k_1}$ and $\mathcal{A}_2\subseteq {[n]\choose k_2}$ with IDs $A_1$ and $A_2$ respectively,  let
\begin{align*}
m(n, A_{1}, A_{2})=&\max \big\{\sum_{j\in [t]\setminus \{1, 2\}}|\mathcal{A}_j|: \mathcal{A}_j\subseteq {[n]\choose k_j} \text { is L-initial and } (\mathcal{A}_1, \dots, \mathcal{A}_t) \text{ is an }\\
& (n, k_1, \dots, k_t)\text{-cross intersecting system, where } k_1+k_3\leq n<k_1+k_2\big\}.
\end{align*}


\begin{proposition}\label{prop2.4}
If $I_1=\{1, n-k_1+2, \dots, n\}$ or $I_1=\{k_t, n-k_1+2, \dots, n\}$ (or $I_2=\{1, n-k_2+2, \dots, n\}$ or $I_2=\{k_t, n-k_2+2, \dots, n\}$), then $\sum_{i=1}^t|\mathcal{F}_i|= \textup{max}\{\lambda_1, \,\, \lambda_2\}$.
\end{proposition}

\begin{proof}
The proofs for $I_1$ and $I_2$ are similar, so we only prove  for $I_1$. Suppose that $I_1=\{1, n-k_1+2, \dots, n\}$ firstly. Since $\{1, n-k_2+2, \dots, n\} \prec I_2$, $I_1\prec I_2$. By  Fact \ref{fact2.7}, we have 
\begin{equation}\label{newdef2}
m(n, I_1)=M(n, k_2, k_3, \dots, k_t).
\end{equation}
Since $(\mathcal{F}_1, \dots, \mathcal{F}_t)$ is extremal,
\begin{align*}
&\sum_{i=1}^t|\mathcal{F}_i|={n-1\choose k_1-1}+m(n, I_1)\\
&~~~~~~~~~~\overset{(\ref{newdef2})}{=}{n-1\choose k_1-1}+M(n, k_2, k_3, \dots, k_t)\\
&~~~~~~\overset{Theorem \ref{HP}}{=}{n-1\choose k_1-1}+\textup{max}\left\{ \sum_{i=2}^t{n-1\choose k_i-1}, \,{n\choose k_2}-{n-k_t\choose k_2}+\sum_{i=3}^t{n-k_t\choose k_i-k_t}  \right\}.
\end{align*}

 Since $k_t\geq 2$,
$${n-1\choose k_1-1}+{n\choose k_2}-{n-k_t\choose k_2}+\sum_{i=3}^t{n-k_t\choose k_i-k_t}<\lambda_2.$$
By (\ref{+1}), $\sum_{i=1}^t|\mathcal{F}_i|=\lambda_1$, as desired.

Next we consider $I_1=\{k_t, n-k_1+2, \dots, n\}$. For each $i\in [3, t]$, since $\mathcal{F}_i$ and $\mathcal{F}_1$ are cross intersecting and $n> k_1+k_3\geq k_1+k_i$,  every element of $\mathcal{F}_i$ must contain $[k_t]$.  Since $\mathcal{F}_i$ and $\mathcal{F}_1$ are cross intersecting for every $i\in [3, t]$, $k_2+k_3\leq n<k_1+k_2$ and $(\mathcal{F}_1, \dots, \mathcal{F}_t)$ is extremal,  we can see that $I_2=\{k_t, n-k_2+2, \dots, n\}$, that is the last element of $\mathcal{R}_2$. Therefore, $\sum_{i=1}^t|\mathcal{F}_i|=\lambda_2$, as desired.
\end{proof}

By Proposition \ref{prop2.4}, the proof of the quantitative part of Theorem \ref{main1} will follow from the following proposition.
\begin{proposition}\label{AS+}
If $\{1, n-k_1+2, \dots, n\}\precneqq   I_1\precneqq \{k_t, n-k_1+2, \dots, n\}$  and $\{1, n-k_2+2, \dots, n\}\precneqq   I_2\precneqq \{k_t, n-k_2+2, \dots, n\}$, then $\sum_{i=1}^t|\mathcal{F}_i|<\max \{\lambda_1, \lambda_2\}$.
\end{proposition}


In order to prove Proposition \ref{AS+}, we need some preparations.

\begin{definition}\label{definition1}
Let $\mathcal{F}_{2,3}\subseteq\mathcal{R}_2$ and $\mathcal{F}_{3}^2\subseteq \mathcal{R}_{3}$ be such that $\mathcal{F}_{2,3}$ and $\mathcal{F}_{3}^2$ are maximal pair families.
\end{definition}

 The following  lemma whose proof will be given in Section \ref{sec5} is crucial.  It tells us that for an extremal $(n, k_1, \dots, k_t)$-cross intersecting system   $(\mathcal{F}_1, \dots, \mathcal{F}_t)$ with $k_1+k_3\leq n <k_1+k_2$, the ID $I_1$ of $\mathcal{F}_1$ is the parity of the ID $I_2$ of $\mathcal{F}_2$.

\begin{lemma}\label{l2.18}
Let $(\mathcal{F}_1, \dots, \mathcal{F}_t)$ be an  extremal L-initial $(n, k_1, \dots, k_t)$-cross intersecting system with IDs $I_1, I_2, \dots, I_t$  of $\mathcal{F}_1, \mathcal{F}_2, \dots, \mathcal{F}_t$ respectively. 
Then $I_1$ and $I_2$  satisfy
\begin{equation}\label{e1}
 I_2\in \mathcal{F}_{2, 3} \,\text{ and }\,I_1 \text{ is the }k_1\text{-parity of }I_2.
 \end{equation}
\end{lemma}


Recall that $k_1>k_2\geq\dots\geq k_t$, and $k_1+k_3\leq n <k_1+k_2$.
 Let $G_{2}\in \mathcal{F}_{2, 3}$. Let $G_1$ be the $k_1$-parity of $G_{2}$. Then from Fact \ref{fact2.7}, we have the following claim.

\begin{remark}\label{remark3.4}
 For any $j\in [3, t]$, $G_1$ and $G_2$ have the same $k_j$-partner.
\end{remark}

{\bf Fixing any $G_{2}\in \mathcal{F}_{2, 3}$, let $G_1$ be the $k_1$-parity of $G_{2}$, and for any $i\in [3, t]$, let  $T_i$ by the $k_i$-partner of $G_{2}$.}
By Claim \ref{>} and Remark \ref{remark2.20+}, we can see that  $\mathcal{L}(T_{3}, k_{3})$,
 \dots, $\mathcal{L}(T_t, k_t)$ are pariwise cross-intersecting. Furthermore, combining  Fact \ref{f2.8} and Remark \ref{remark3.4},  we conclude that
\begin{align}\label{equation4}
m(n, G_1, G_{2})&=\sum_{i=3}^t|\mathcal{L}(T_i, k_i)|.
\end{align}
Now we are ready to express $M(n, k_1, \dots, k_t)$ in terms of a function of $G_{2}$.

\begin{remark}\label{def3.9}
Let $G_{2}\in \mathcal{F}_{2, 3}$,  let $G_1$ be the $k_1$-parity of $G_{2}$ and for any $i\in [3, t]$, let  $T_i$ be the $k_i$-partner of $G_{2}$. Define
\begin{equation}\label{equation3+}
g(G_{2})=\sum_{i=1}^{2}|\mathcal{L}(G_i, k_i)|+\sum_{i=3}^t|\mathcal{L}(T_i, k_i)|.
\end{equation}
Then by Lemma \ref{l2.18} and (\ref{equation4}),
\begin{equation}\label{new1+}
M(n, k_1, \dots, k_t)=\max_{G_2\in \mathcal{F}_{2, 3}}\{g(G_2)\}.
\end{equation}
\end{remark}

\begin{definition}\label{def3.11}
 For each $j\in [k_2-1]$, let $\mathcal{R}_{2, j}=\{R\in \mathcal{R}_2: [n-j+1, n]\subset R\}$ and $\mathcal{R}_{2}(j)=\{R\setminus [n-j+1, n]: R\in \mathcal{R}_{2, j}\}$. Denote $\mathcal{F}_{2, 3}(j)=\{R\setminus [n-j+1, n]: R\in \mathcal{F}_{2, 3} \cap \mathcal{R}_{2, j}\}$. For any $j\in [k_2-1]$ and any $R\in \mathcal{F}_{2, 3}\cap \mathcal{R}_{2, j}$, we define $g(R\setminus [n-j+1, n])=g(R)$.
\end{definition}

We will prove several key lemmas  to show the `local unimodality' of $g(G_{2})$ in Section \ref{sec5}. Before stating these crucial lemmas, we need to introduce some definitions.

Let $\mathcal{A}\subset {[n]\choose k}$ be a family and $c\in [k]$. We say that $\mathcal{A}$ is $c$-{\it sequential} if there are $A\subset [n]$ with $|A|=k-c$ and $a\geq \max A$ (For a set $A\subset [n]$,  denote $\max A=\max \{x: x\in A\}$ and $\min A=\min \{x: x\in A\}$) such that $\mathcal{A}=\{A\sqcup \{a+1, \dots, a+c\}, A\sqcup \{a+2, \dots, a+c+1\}, \dots, A\sqcup \{b-c+1, \dots, b\}\}$, then we say that  $\mathcal{A}$ is $c$-sequential, write $A_1\overset{c}{\prec} A_2\overset{c}{\prec}\cdots\overset{c}{\prec}A_{b-a-c+1}$, where $A_1=A\sqcup \{a+1, \dots, a+c\}$, $A_2=A\sqcup \{a+2, \dots, a+c+1\}$,$\dots$, $A_{b-a-c+1}=A\sqcup \{b-c+1, \dots, b\}$,
 we also say that $A_i$ and $ A_j$ are $c$-sequential.
 In particular, if $l_2=l_1+1$, we write $A_{l_1}\overset{c}{\prec}  A_{l_2}$; if $\max A_{l_2}=n$, write $A_{l_1}\overset{c}{\longrightarrow}  A_{l_2}$. Note that if $|\mathcal{A}|=1$, then $\mathcal{A}$ is $c$-sequential for any $c\in [k]$. Let $\mathcal{F}$ be a family and $F_1, F_2\in \mathcal{F}$. If $F_1\precneqq F_2$ and there is no $F'\in \mathcal{F}$ such that $F_1\precneqq F' \precneqq F_2$, then we say that $F_1<F_2$ in $\mathcal{F}$, or $F_1<F_2$ simply if there is no confusion.

We will prove the following crucial lemmas in Sections \ref{sec6} and \ref{sec7}. They show that function $g(R)$ has local unimodality.

\begin{lemma}\label{c2.27}
For any $j\in [0, k_{2}-1]$, let $1\leq c\leq k_{2}-j$ and $F_{2}, G_{2}, H_{2}\in\mathcal{F}_{2, 3}(j)$ with $F_{2}\overset{c}{\prec}G_{2}\overset{c}{\prec}H_{2}$.
Then $g(G_{2})\geq g(F_{2})$ implies $g(H_{2})>g(G_{2})$.
\end{lemma}

\begin{lemma}\label{c2.28}
Let $4\leq j\leq k_{2}+1$ and $F_{2}, G_{2}, H_{2}\in\mathcal{F}_{2, 3}$ with $[2, j]=F_{2}\setminus [n-\ell(F_{2})+1, n]$, $[2, j-1]=G_{2}\setminus [n-\ell(G_{2})+1, n]$ and $[2, j-2]=H_{2}\setminus [n-\ell(H_{2})+1, n]$.  Then $g([2, j-1])\geq g([2, j])$ implies $g([2, j-2])>g([2, j-1])$.
\end{lemma}

\begin{lemma}\label{c2.29}
Let $k_t+2\leq j\leq k_t+k_{2}-1$ and $F_{2}, G_{2}, H_{2}\in\mathcal{F}_{2, 3}$ with $[k_t, j]=F_{2}\setminus [n-\ell(F_{2})+1, n]$, $[k_t, j-1]=G_{2}\setminus [n-\ell(G_{2})+1, n]$ and $[k_t, j-2]=H_{2}\setminus [n-\ell(H_{2})+1, n]$.
Then $g([k_t, j-1])\geq g([k_t, j])$ implies $g([k_t, j-2])>g([k_t, j-1])$.
\end{lemma}

\begin{claim}\label{clm2.25}
Let $A$ be a $k_2$-subset of $[n]$ with $\ell(A)=p$. Suppose $A=\{x_1, \dots, x_{k_2-p}, n-p+1, \dots, n\}$ and $A'=A\setminus \{x_{k_2-p}\}\cup[n-p, n]$. If $A\in \mathcal{F}_{2, 3}$, then $A'\in \mathcal{F}_{2, 3}$.
\end{claim}

\begin{proof}
Let $T$ and $T'$ be the partners of $\{x_1, \dots, x_{k_2-p}\}=A\setminus A^{\rm t}$ and $\{x_1, \dots, x_{k_2-p-1}\}=A'\setminus A'^{\rm t}$ respectively. Since $A\in \mathcal{F}_{2, 3}$, $|T|\leq k_3$. Thus, $|T'|\leq |T|\leq k_3$. Let $B=T'\cup \{n-k_3+|T'|+1, \dots, n\}$ if $|T'|<k_3$, otherwise, let $B=T'$. Then $B$ is the $k_3$-partner of $\{x_1, \dots, x_{k_2-p-1}\}$. Fact \ref{fact2.10} implies that $(A', B)$ is maximal, thus $A'\in \mathcal{F}_{2, 3}$.
\end{proof}

\begin{claim}\label{clm2.26}
Let $A$ be a $k_2$-subset of $[n]$ with $\ell(A)=p\geq 1$. Suppose that $A=\{x_1, \dots, x_{k_2-p}, n-p+1, \dots, n\}$ and $A'=A\setminus \{n-p+1\}\cup\{x_{k_2-p}+1\}\cup [n-p+2, n]$ (if $p=1$, then $[n-p+2, n]=\emptyset$). If $A\in \mathcal{F}_{2, 3}$ and $\{1, n-k_2+2, \dots, n\}\precneqq A$, then $A'\in \mathcal{F}_{2, 3}$.
\end{claim}

\begin{proof}
Let $T$ and $T'$ be the partners of $\{x_1, \dots, x_{k_{2}-p}\}=A\setminus A^{\rm t}$ and $\{x_1, \dots, x_{k_{2}-p},\\ x_{k_{2}-p}+1\}=A'\setminus A'^{\rm t}$ respectively. Then $|T|=|T'|$ and $\max T'-\max T=1$.
Since $A\in \mathcal{F}_{2, 3}$, by Definition \ref{definition1}, there is $B\in \mathcal{F}_{3}^{2}$ such that $(A, B)$ is maximal. By Fact \ref{fact2.5}, $T=B\setminus \{n-\ell(B)+1, \dots, n\}$, $B=T\cup \{n-k_{3}+|T|+1, \dots, n\}$ and $\max T<n-k_{3}+|T|$.
We claim that $\max T<n-k_{3}+|T|-1$. Since otherwise $\max T=n-k_{3}+|T|-1$, this implies that
$|\{x_1, \dots, x_{k_{2}-p}\}|+|T|=n-k_{3}+|T|$ and then
$$|A|+|B|\geq |\{x_1, \dots, x_{k_{2}-p}\}|+1+ |T|+|\{n-k_{3}+|T|+1, \dots, n\}|=n+1,$$
this is a contradiction to $n\geq k_{2}+k_{3}\,\, (=|A|+|B|)$.
Let $B'=T'\cup \{n-k_{3}+|T'|+1, \dots, n\}$ if $|T'|<k_{3}$, and $B'=T'$ otherwise.
Thus,
$$\max T'=\max T+1<n-k_{3}+|T|= n-k_{3}+|T'|.$$
Therefore, $T'=B'\setminus \{n-\ell(B')+1, \dots, n\}$.
Trivially, $\{x_1, \dots, x_{k_{2}-p}, x_{k_{2}-p}+1\}=A'\setminus \{n-\ell(A')+1, \dots, n\}$.
By Fact \ref{fact2.5} again, $(A', B')$ is maximal and since $\{1, n-k_2+2, \dots, n\}\precneqq A$, we have $A'\in \mathcal{F}_{2, 3}$, as desired.
\end{proof}

\begin{definition}\label{definition2}
Let $\mathcal{F}_{1,3}\subseteq\mathcal{R}_1$ and $\mathcal{F}_{3}^1\subseteq \mathcal{R}_{3}$ such that $\mathcal{F}_{1,3}$ and $\mathcal{F}_{3}^1$ are maximal pair families.
By Proposition \ref{prop2.9}, $\mathcal{F}_{3}^2\subseteq \mathcal{F}_{3}^1$. Let $\mathcal{F}'_{1,3}$ be the subfamily of $\mathcal{F}_{1,3}$ such that $\mathcal{F}'_{1,3}$ and $\mathcal{F}_{3}^2$ are maximal pair families.
\end{definition}

\begin{obser}\label{o}
Since $k_3\geq k_t\geq 2$ and $n\geq k_1+k_3$, by Fact \ref{fact2.5}, it is easy to see that $\{1, n-k_2+1,\dots, n\}, \{2, \dots, k_2+1\}, \{k_t, k_t+1, \dots, k_t+k_2-1\}, \{k_t, n-k_2+1,\dots, n\}$  are in $\mathcal{F}_{2, 3}$. As a consequence of Claim \ref{clm2.25} and Claim \ref{clm2.26}, we have the following observation.
$\{2, \dots, k_2, n\}, \{2, \dots, k_2-1, n-1, n\}, \dots, \{2, n-k_2+2, \dots, n\}, \{k_t, k_t+1, \\\dots, k_t+k_2-2, n\},\{k_t, k_t+1, \dots, k_t+k_2-3, n-1, n\}, \dots, \{k_t, n-k_2+2, \dots, n\}$ are in $\mathcal{F}_{2, 3}$ as well.
$\mathcal{F}'_{1, 3}$ contains $\{1, n-k_1+2, \dots, n\},  \{2, \dots, k_1+1\}, \{2, \dots, k_1, n\}, \{2, \\ \dots, k_1-1, n-1, n\}, \dots, \{2, n-k_1+2, \dots, n\},   \{k_t, k_t+1, \dots, k_t+k_1-1\},\{k_t, k_t+1, \dots, k_t+k_1-2, n\},\{k_t, k_t+1, \dots, k_t+k_1-3, n-1, n\}, \dots, \{k_t, n-k_1+2, \dots, n\}$. Note that $\mathcal{F}'_{1, 3}\subseteq \mathcal{L}(\{k_t, n-k_1+2, \dots, n\}, k_1)$ since we require that $(n, k_1, \dots, k_t)$-cross intersecting systems is non-empty.
\end{obser}

Assuming that Lemmas \ref{l2.18}, \ref{c2.27}, \ref{c2.28} and \ref{c2.29} hold, 
we are going to complete the proof of Proposition \ref{AS+}. We will prove 
Lemmas \ref{l2.18}, \ref{c2.27}, \ref{c2.28} and \ref{c2.29} in Sections \ref{sec5}, \ref{sec6} and \ref{sec7}.

\begin{definition}
For a family $\mathcal{G}\subseteq \mathcal{F}_{2, 3}$, denote $g(\mathcal{G})=\max \{g(G): G\in \mathcal{G}\}$.
\end{definition}


\begin{proof}[Proof of Proposition \ref{AS+}]
By Observation  \ref{o} and   Lemma \ref{c2.27}, we have
\begin{equation}\label{equation10}
g(\mathcal{F}_{2, 3})=\max \{g([2, k_{2}+1]), g([k_t, k_t+1,  k_t+k_{2}-1]), g(\mathcal{F}_{2, 3}(1))\}.
\end{equation}

Applying  Lemma \ref{c2.27} and Observation  \ref{o} repeatedly, we have
\begin{align}\nonumber
&g(\mathcal{F}_{2, 3}(1))=\max \{g([2, k_{2}]), g([k_t, k_t+k_{2}-2]), g(\mathcal{F}_{2, 3}(2))\},\\ \nonumber
&g(\mathcal{F}_{2, 3}(2))=\max \{g([2, k_{2}-1]), g([k_t, k+k_{2}-3]), g(\mathcal{F}_{2, 3}(3))\},\\\nonumber
&\quad \vdots\\ \label{fr}
&g(\mathcal{F}_{2, 3}(k_{2}-1))<\max \{g(\{1\}),  g(\{k_t\})\}.
\end{align}
By Lemma \ref{c2.28}, we have
\begin{align}\label{dd}
&\max\{g([2, k_{2}+1]), g([2,  k_{2}]), \dots, g(\{2, 3\}), g(\{2\})\} \nonumber \\
&=\max\{g([2,  k_{2}+1]), g(\{2\})\},
\end{align}
and
\begin{equation}\label{dd+}
g(\{2\})<\max \{g(\{1\}),  g(\{k_t\})\}.
\end{equation}
By Lemma \ref{c2.29}, we have
\begin{align}\label{d}
&\max\{g([k_t, k_t+k_{2}-1]), g([k_t, k_t+k_{2}-2]), \dots,   g(\{k_t, k_t+1\})\} \nonumber \\
&=\max\{g([k_t, k_t+k_{2}-1]),  g(\{k_t, k_t+1\}) \},
\end{align}
and
\begin{equation}\label{d+}
g(\{k_t, k_t+1\})<\max\{g([k_t, k_t+k_{2}-1]),  g(\{k_t\}) \}.
\end{equation}


In Section \ref{sec7}, we will prove the following proposition.

\begin{proposition}\label{claim3.20}
\begin{align}\label{equation13}
g([2, k_{2}+1])&<\max\{g(\{1\}), g([2, k_{2}])\}\\\label{equation14}
g([k_t, k_t+k_{2}-1])&<\max\{ g(\{k_t-1\}), g([k_t, k_t+k_{2}-2])\}.
\end{align}
\end{proposition}

Combining (\ref{equation13}), (\ref{dd}) and (\ref{dd+}), we have
\begin{equation}\label{222}
g([2, k_{2}+1])<\max \{  g(\{1\}), g(\{k_t\})    \}.
\end{equation}

Combining (\ref{equation14}), (\ref{d}) and (\ref{d+}), we have
\begin{equation}\label{ttt}
g([k_t, k_t+k_{2}-1])<\max \{ g(\{k_t\}),  g(\{k_t-1\})\} \leq\max \{  g(\{1\}), g(\{k_t\})    \}.
\end{equation}

Combining (\ref{equation10}), (\ref{fr}), (\ref{dd}), (\ref{dd+}), (\ref{d}), (\ref{d+}),(\ref{222}) and (\ref{ttt}), we have
\begin{equation*}
g(\mathcal{F}_{2, 3})< \max \{ g(\{1\}),  g(\{k_t\})\}.
\end{equation*}

Recall (\ref{new1+}), we obtain
\begin{align}\label{equation3}
M(n, k_1, \dots, k_t)&=\max_{G_{2} \in \mathcal{F}_{2, 3}}g(G_{2})<\max \{ g(\{1\}),  g(\{k_t\})\}.
\end{align}
This completes the proof of Proposition \ref{AS+}.
\end{proof}

We apply Propositions \ref{prop2.4} and \ref{AS+} to prove Theorem \ref{main1}.
\begin{proof}[Proof of Theorem \ref{main1}]
Propositions \ref{prop2.4} and \ref{AS+} imply the quantitative part of Theorem \ref{main1}.
Now we show that extremal $(n, k_1, \dots, k_t)$-cross intersecting systems with $k_1+k_3\leq n <k_1+k_2$ must be isomorphic to Constructions \ref{con1} or \ref{con2}.
By Claim \ref{>}, Propositions \ref{prop2.4} and \ref{AS+}, and Theorem \ref{kk}, we conclude that: if $(\mathcal{F}_1, \dots, \mathcal{F}_t)$ is an $(n, k_1, \dots, k_t)$-cross intersecting system with $\sum_{i=1}^t|\mathcal{F}_i|=\max \{\lambda_1, \lambda_2\}$, then either for each $i\in [t]$, $|\mathcal{F}_i|={n-1\choose k_i-1}$ or $|\mathcal{F}_1|={n\choose k_1}-{n-k_t\choose k_1}$, $|\mathcal{F}_2|={n\choose k_2}-{n-k_t\choose k_2}$ and $|\mathcal{F}_i|={n-k_t\choose k_i-k_t}$ holds for each $i\in [3, t]$.
If the previous happens, then $(\mathcal{F}_1, \mathcal{F}_3,\dots, \mathcal{F}_t)$ is an $(n, k_1, k_3, \dots, k_t)$-cross intersecting system with
$\sum_{i=1, i\ne 2}^t|\mathcal{F}_i|=\sum_{i=1, i\ne 2}^t{n-1\choose k_i-1}$
 and
 $(\mathcal{F}_2, \mathcal{F}_3,\dots, \mathcal{F}_t)$ is an $(n, k_2, k_3, \dots, k_t)$-cross intersecting system with
$\sum_{i=2}^t|\mathcal{F}_i|=\sum_{i=2}^t{n-1\choose k_i-1}$.
By Theorem \ref{HP},  $(\mathcal{F}_1, \dots, \mathcal{F}_t)$ is isomorphic to $(\mathcal{G}_1, \dots, \mathcal{G}_t)$ which is defined in Construction \ref{con1}.
If the later happens, then $(\mathcal{F}_1, \mathcal{F}_3,\dots, \mathcal{F}_t)$ is an $(n, k_1, k_3, \dots, \\ k_t)$-cross intersecting system with
$\sum_{i=1, i\ne 2}^t|\mathcal{F}_i|={n\choose k_1}-{n-k_t\choose k_1}+\sum_{i=3}^t{n-k_t\choose k_i-k_t}$
 and
 $(\mathcal{F}_2, \mathcal{F}_3,\dots, \mathcal{F}_t)$ is an $(n, k_2, k_3, \dots, k_t)$-cross intersecting system with
$\sum_{i=2}^t|\mathcal{F}_i|={n\choose k_2}-{n-k_t\choose k_2}+\sum_{i=3}^t{n-k_t\choose k_i-k_t}$.
By Theorem \ref{HP},  $(\mathcal{F}_1, \dots, \mathcal{F}_t)$ is isomorphic to $(\mathcal{H}_1, \dots, \mathcal{H}_t)$ which is defined in Construction \ref{con2}.
 This completes the proof of Theorem \ref{main1}.
 \end{proof}

We owe the proofs of Lemmas \ref{l2.18}, \ref{c2.27}, \ref{c2.28}, \ref{c2.29} and Proposition \ref{claim3.20}. Before giving their proofs, we list some results obtained in \cite{HP} for non-mixed type in the next section. 
Figure 1 is the flow chart of the proofs of  the main theorem and lemmas.
\begin{figure}[htbp]
\centering
\includegraphics[scale=0.37]{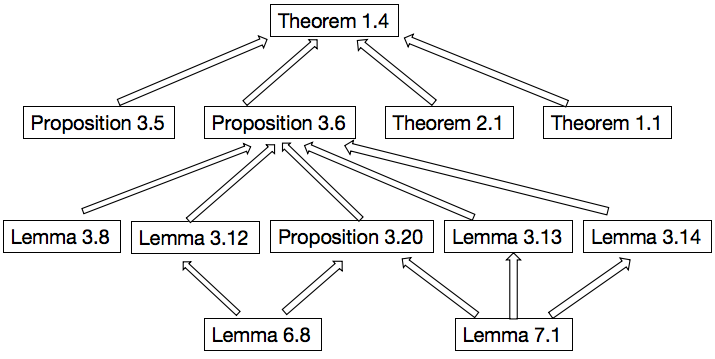}
\caption{Flow chart of the proofs}
\label{Fig-Liucheng}
\end{figure}

\section{Results of non-mixed type}\label{sec4}
In \cite{HP},  we worked on non-mixed type: Let $t\geq 2$, $k_1\geq k_2\geq \cdots \geq k_t$ and $n\geq k_1+k_2$  and families $\mathcal{A}_1\subset{[n]\choose k_1}, \mathcal{A}_2\subset{[n]\choose k_2}, \dots, \mathcal{A}_t\subset{[n]\choose k_t}$ be non-empty pairwise cross-intersecting (not freely). Let $R$ be the ID of $\mathcal{A}_1$, and $T$  be the partner of $R$. In the proof of Theorem \ref{HP}, one important ingredient is that by Proposition \ref{p2.7}, $\sum_{i=1}^t{|\mathcal{A}_i|}$ can be bounded by a function of $R$ as in the following lemma.

\begin{lemma}\cite{HP}\label{c2.22}
Let  $k_1\geq k_2\geq \dots \geq k_t$, $n\geq k_1+k_2$ and $(\mathcal{A}_1, \mathcal{A}_2, \dots, \mathcal{A}_t)$ be a non-empty L-initial $(n, k_1, k_2,
\dots, k_t)$-cross intersecting system with $|\mathcal{A}_1|\geq {n-1\choose k_1-1}$. Let $R$ be the ID of $\mathcal{A}_1$ and $T$ be the partner of $R$. Then
\begin{align}\label{new}
\sum_{j=1}^t|\mathcal{A}_j|\leq |\mathcal{L}(R, k_1)| +\sum_{j=2}^t |\mathcal{L}(T, k_j)| =: f(R).
\end{align}
\end{lemma}

Another crucial part in the proof of Theorem \ref{HP} is to show  local unimodality of $f(R)$. Let $R$ and $R'$  be  $k_1$-subsets of $[n]$ satisfying $R\prec R'$. In order to analyze $f(R')-f(R)$, two related functions are defined in \cite{HP} as follows.

\begin{definition}\label{ab-}
Let $t\geq 2$, $k_1, k_2, \dots, k_{t}$ be positive integers with $k_1\geq k_2\geq \dots \geq k_{t}$ and $n\geq k_1+k_2$.
Let $R$ and $R'$  be  $k_1$-subsets of $[n]$ satisfying $R\prec R'$ and let $T$ and $T'$ be the partners of $R$ and $R'$ respectively. We  define
\begin{align}\label{eq6}
&\alpha(R, R'):=|\mathcal{L}(R', k_1)|-|\mathcal{L}(R, k_1)|,\\ \label{eq7}
&\beta(R, R'):=\sum_{j=2}^t\big(|\mathcal{L}(T, k_j)|-|\mathcal{L}(T', k_j)|\big).
\end{align}
\end{definition}
It is easy to see that
\begin{align*}
f(R')-f(R)=\alpha(R, R')-\beta(R, R').
\end{align*}
Let
$$\mathcal{R}_1=\{ R\in {[n]\choose k_1}:\{1, n-k_1+2, \dots, n\}\prec R\prec \{k_{t}, n-k_1+2, \dots, n\}\}.$$

In order to show local unimodality of $f(R)$, we  proved the  following  results in  \cite{HP}. These results give us some foundation in showing local unimodality of $g(G_2)$ (recall (\ref{equation3+})).

\begin{proposition}\cite{HP}\label{clm23}
Let $F<G\in \mathcal{R}_1$ and $\max G=q$. Then
 $\beta(F, G)=\sum_{j=2}^t{n-q\choose k_j-(q-k_1)}$. If $n= k_1+k_j$ holds for any $j\in [t]\setminus \{1\}$, then $\beta(F, G)=\sum_{j=2}^t 1=t-1$; otherwise, we have $\beta(F, G)\geq 0$ and $\beta(F, G)$ decrease as $q$ increase.
\end{proposition}



\begin{lemma}\cite{HP}\label{coro22}
Let $c\in [k_1]$ and $F, G, F', G'\in \mathcal{R}_1$. If $F, G$ are $c$-sequential, $F', G'$ are $c$-sequential and $\max F=\max F', \max G=\max G'$, then $\alpha (F, G)=\alpha (F', G')$ and $\beta (F, G)=\beta (F', G')$.
\end{lemma}

\begin{definition}\label{definition}
For $k\in [k_{1}-1]$, let  $\mathcal{R}_{1, k}=\{R\in \mathcal{R}_{1}: [n-k+1, n]\subset R\}$, and
$\mathcal{R}_{1}(k)=\{R\setminus [n-k+1, n]: R\in \mathcal{R}_{1, k}\}$. In addition, we will write $\mathcal{R}_{1}(0)=\mathcal{R}_{1}$.
\end{definition}

\begin{remark}
When we consider $\mathcal{R}_{i}(k)$, we regard the ground set as $[n-k]$.
For $R\in \mathcal{R}_{i}(k)$, we write $f(R)$ simply, in fact, $f(R)=f(R\cup [n-k+1, n])$.
\end{remark}

\begin{lemma}\cite{HP}\label{clm27-}
Let $1\leq j\leq k_1-1$ and $1\leq d\leq k_1-j$. Let $F, H, F', H'\in \mathcal{R}_1(j)$ and $F$, $H$ are $d$-sequential, $F'$, $H'$ are $d$-sequential. If $\max F=\max F'$, then
$\alpha(F, H)=\alpha(F', H')$ and $\beta(F, H)=\beta(F', H')$.
\end{lemma}


The following two lemmas confirm that $f(R)$ has local unimodality.

\begin{lemma}\cite{HP}\label{clm28}
Suppose that if $t=2$, then $n>k_1+k_t$.
For any $j\in [0, k_1-1]$, let $1\leq c \leq k_1-j$ and $F$, $G$, $H$ be contained in $\mathcal{R}_1(j)$ with $F\overset{c}{\prec}G\overset{c}{\prec}H$. If $f(G)\geq f(F)$, then $f(H)> f(G)$.
\end{lemma}

\begin{lemma}\cite{HP}\label{clm30}
Suppose that if $t=2$, then $n>k_1+k_t$.
Let $1\leq m \leq k_t$ and $m+1 \leq j \leq m+k_1-1$. If $f_1([m, j])\leq f_1([m, j-1])$, then $f([m, j-1])< f([m, j-2])$.
\end{lemma}

\section{Verify Parity: Proof of Lemma \ref{l2.18}}\label{sec5}

In this section,  we will give the proof  of Lemma \ref{l2.18}. The proof is divided into  two cases.

\textbf{Case 1: $I_1\prec I_2$. }

By Fact \ref{fact2.7} , Fact \ref{f2.8}, Remark \ref{remark2.20+} and $(\mathcal{F}_1, \dots, \mathcal{F}_t)$ is extremal, we have that for each $i\in [3, t]$, $I_i$ is the $k_i$-partner of $I_2$. By Fact \ref{f2.11}, for all $i\in [4, t]$, we have
\begin{equation}\label{e2}
I_i\prec I_3 \text{ or } I_3 \text{ is the $k_3$-parity of } I_i .
\end{equation}

\begin{claim}\label{clm2.16}
$(I_2, I_3)$ is maximal.
\end{claim}

\begin{proof}
 Assume on the contrary that $(I_2, I_3)$ is not maximal. Let $F_{2, i}$ be the $k_2$-partner of $I_i$ for each $i\in [3, t]$, by Fact \ref{f2.10}, $(I_{2, i}, I_i)$ is maximal. Since $(I_2, I_3)$ is not maximal,   $I_2\precneqq I_{2, 3}$. By (\ref{e2}) and Fact \ref{fact2.7}, we have $I_{2, 3}\prec I_{2, j}$ for all $j\in [4, t]$. This implies that for $i\in [3, t]$, $\mathcal{L}(I_{2, 3}, k_2)$ and $\mathcal{L}(I_i, k_i)$ are cross intersecting. Further more, $(\mathcal{F}_1, \mathcal{L}(I_{2, 3}, k_2), \mathcal{F}_3, \dots, \mathcal{F}_t )$ is an $(n, k_1, \dots, k_t)$-cross intersecting system. However, $I_2\precneqq I_{2, 3}$ implies $\mathcal{F}_2 \subsetneqq \mathcal{L}(I_{2, 3}, k_2)$, this is a contradiction to the assumption that $(\mathcal{F}_1, \mathcal{F}_2, \dots, \mathcal{F}_t)$ is extremal.
\end{proof}
Claim \ref{clm2.16} tells us that $I_2 \in \mathcal{F}_{2, 3}$.
By Proposition \ref{prop2.9}, $\mathcal{F}_{1, 3}$ is the $k_1$-parity of $\mathcal{F}_{2, 3}$. Then there exists $I'_1\in \mathcal{F}_{1, 3}$ such that $I'_1$ is the $k_1$-parity of $I_2$.
If $I_1=I'_1$, then $I_1$ and $I_2$ satisfy (\ref{e1}), as desired. So we may assume  $I_1\ne I'_1$.
By Proposition \ref{prop2.9}, $(I'_1, I_3)$ is maximal. Since $\mathcal{F}_1$ and $\mathcal{F}_3$ are cross intersecting, $I_1\precneqq I'_1$.
Let $I'_i$ be the $k_i$-partner of $I'_1$ for each $i\in [3, t]$. It follows from Fact \ref{fact2.7} that $I'_i=I_i$, $i\in [3, t]$. Therefore, $(\mathcal{L}(I'_1, k_1), \mathcal{F}_2,  \dots, \mathcal{F}_t )$ is an $(n, k_1, \dots, k_t)$-cross intersecting system. However, $I_1\precneqq I'_1$ implies $\mathcal{F}_1 \subsetneqq \mathcal{L}(I'_1, k_1)$, this is a contradiction to the assumption that $(\mathcal{F}_1, \mathcal{F}_2, \dots, \mathcal{F}_t)$ is extremal.

\textbf{Case 2: $I_2\precneqq I_1$.}

Using a similar argument to Case 1, we conclude the following three properties:

(a) $I_i$ is the $k_i$-partner of $I_1$ for each $i\in [3, t]$;

(b) for all $i\in [4, t]$, $I_i\prec I_3$ or $I_3$ is the $k_3$-parity of  $I_i$;

(c) $(I_1, I_3)$ is maximal and $I_1 \in \mathcal{F}_{1, 3}$.

If $(I_2, I_3)$ is maximal, then $I_2 \in \mathcal{F}_{2, 3}$, and Proposition \ref{prop2.9} implies that
$I_1$ is the $k_1$-parity of $I_2$. However, $I_2\precneqq I_1$ implies that  $I_1$ is not the $k_1$-parity of $I_2$, a contradiction. Therefore $(I_2, I_3)$ is not maximal. Since $(\mathcal{F}_1, \mathcal{F}_2, \dots, \mathcal{F}_t)$ is extremal, combining (b), Fact \ref{fact2.7} and Fact \ref{f2.8}, we have that $I_2$ is the $k_2$-partner of $I_3$. Therefore, Fact \ref{f2.10} implies that the $k_3$-partner $I'_3$ of $I_2$ is such that $(I_2, I'_3)$ is maximal. So $I_2\in \mathcal{F}_{2, 3}$.
By Proposition \ref{prop2.9}, there exists $I'_1\in \mathcal{F}_{1, 3}$ such that $I'_1$ is the $k_1$-parity of $I_2$.
For $i\in [3, t]$, let $I'_i$ be the $k_i$-partner of $I_2$. It follows from Facts \ref{fact2.7} and  \ref{f2.8} that $I'_i$ is maximal to $I'_1$ for all $i\in [3, t]$. Combining with Remark \ref{remark2.20+}, we conclude that $(\mathcal{L}(I'_1, k_1), \mathcal{F}_2, \mathcal{L}(I'_3, k_3), \dots, \mathcal{L}(I'_t, k_t))$ is an $(n, k_1, \dots, k_t)$-cross intersecting system. Since $(\mathcal{F}_1, \dots, \mathcal{F}_t)$ is extremal,
\begin{equation}
|\mathcal{L}(I'_1, k_1)|+|\mathcal{F}_2|+\sum_{i=3}^t|\mathcal{L}(I'_i, k_i)|\leq \sum_{i=1}^t|\mathcal{F}_i|.\nonumber
\end{equation}
This implies that
\begin{equation}\label{e4}
\sum_{i=1, i\ne 2}^t|\mathcal{L}(I'_i, k_i)|\leq \sum_{i=1, i\ne 2}^t|\mathcal{F}_i|.
\end{equation}


Let $H_1=I_1\setminus I_1^{\rm t}$ and  $g=|H_1|$.
Recall that $I_2$ is the $k_2$-partner of $I_3$, $(I_2, I_3)$ is not maximal and $(I_1, I_3)$ is maximal. It follows from Facts \ref{fact2.10} and \ref{fact2.4+} that
$$k_2< g.$$
Let $H_2=\{x_1, \dots, x_i, x_{i+1}, \dots, x_{k_2}\}$ be the first $k_2$ elements of $H_1$.
Let $i\in [0, k_2-1]$ be the  subscript such that $x_i+2\leq x_{i+1}$ and $x_j+1=x_{j+1}$ for all $j\in [i+1, k_2-1]$, where $i=0$ if $x_{j+1}=x_j+1$ for all $j\in [k_2-1]$.
Since $I_2$ is the $k_2$-partner of $I_3$ and $(I_1, I_3)$ is maximal, we can see that if $i<k_2-1$, then $I_2=\{x_1, \dots, x_i, x_{i+1}-1, n-k_2+i+2, \dots, n\}$,  if $i=k_2-1$, then $I_2=\{x_1, \dots, x_{k_2-1}, x_{k_2}-1\}$. Since $I'_1$ is the $k_1$-parity of $I_2$ and $k_1>k_2$, we get $\ell(I'_1)>0$ and
\[I'_1=\{x_1, \dots, x_i, x_{i+1}-1, n-k_1+i+2, \dots, n\}, \,where  \,\, i\leq k_2-1.\]

\begin{claim}\label{clm2.20}
$I'_1\precneqq I_1\precneqq I''_1:=\{x_1, \dots,  x_{i+1}, n-k_1+i+2,  \dots, n\}.$
\end{claim}

\begin{proof}
Obviously, $I'_1\precneqq I_1$.
We are going to prove $I_1\precneqq I''_1$.
Notice that $\{x_1, \dots, x_i,\\  x_{i+1}\}\subset I_1$ and $|I''_1|=k_1$, these imply $I_1 \prec I''_1$.
If $I_1=I''_1$, then $H_1=\{x_1, \dots, x_{i+1}\}$. So $g=i+1\leq k_2$ since $i\leq k_2-1$, this is a contradiction to $k_2<g$.
\end{proof}

Since $(\mathcal{F}_1, \mathcal{F}_2, \dots, \mathcal{F}_t)$ is an $(n, k_1, k_2, \dots, k_t)$-cross intersecting system with $k_1+k_3\leq n<k_1+k_2$, $(\mathcal{F}_1, \mathcal{F}_3, \dots, \mathcal{F}_t)$ is an $(n, k_1, k_3, \dots, k_t)$-cross intersecting system with $n\geq k_1+k_3$.
Denote
\begin{align*}
m' &=\max \big\{\sum_{i=1}^t|\mathcal{G}_i|:  (\mathcal{G}_1, \mathcal{G}_3, \dots, \mathcal{G}_t) \text{ is an L-initial } (n, k_1, k_3, \dots, k_t)\text{-cross intersecting}\\
&\quad\quad\quad\quad \text{system with $n\geq k_1+k_3$, the ID} \ {G}_1 \ \text{of} \  \mathcal{G}_1 \ \text{satisfies} \  I'_1\prec G_1\prec I''_1 \big\}.
\end{align*}

Suppose that $(\mathcal{L}(G_1, k_1), \mathcal{L}(G_3, k_3), \dots, \mathcal{L}(G_t, k_t))$ is an $(n, k_1, \dots, k_t)$-cross intersecting system with $n\geq k_1+k_3$ such that $m' =\sum_{i=1, i\ne 2}^t|\mathcal{G}_i| $ and $I'_1\prec G_1\prec I''_1$.  Since $I'_1$ is the $k_1$-parity of $I_2$ and  $I'_1\prec G_1\prec I''_1$,
we have that either $I_2 \prec G_1$ (if $G_1\ne I'_1$) or $G_1$ is the $k_1$-parity of $I_2$ (if $G_1= I'_1$). Thus, $(\mathcal{L}(G_1, k_1), \mathcal{F}_2,  \mathcal{L}(G_3, k_3), \dots, \mathcal{L}(G_t, k_t))$ is an $(n, k_1, \dots, k_t)$-cross intersecting system as well.
Since $(\mathcal{F}_1, \mathcal{F}_2, \dots, \mathcal{F}_t)$ is extremal,  we have $\sum_{i=1, i\ne 2}^t|\mathcal{F}_i|\ge \sum_{i=1, i\ne 2}^t|\mathcal{G}_i|=m'$. Therefore,
\begin{equation}\label{contra}
m'=\sum_{i=1, i\ne 2}^t|\mathcal{F}_i|.
\end{equation}
To complete the proof of Lemma \ref{l2.18}, we only need to prove the following claim since it makes a contradiction to (\ref{contra}) and Claim \ref{clm2.20}.

\begin{claim}\label{clm2.21}
Let $(\mathcal{G}_1, \mathcal{G}_3, \dots, \mathcal{G}_t)$ be an L-initial $(n, k_1, k_3, \dots, k_t)$-cross intersecting system with $n\geq k_1+k_3$ and  $G_1$ be the ID of $\mathcal{G}_1$ satisfying $I'_1\prec G_1\prec I''_1$.  Then $\sum_{i=1, i\ne 2}^t|\mathcal{G}_i|=m'$ if and only if  $G_1=I'_1$ or  $G_1=I''_1$.
\end{claim}

\begin{proof}
For each $i\in [3, t]$, let $T_i$ be the $k_i$-partner of $G_1$. By Remark \ref{r2.6} and Lemma \ref{c2.22}, we have
$$\sum_{i=1, i\ne 2}^t|\mathcal{G}_i|\leq|\mathcal{L}(G_1, k_1)|+\sum_{i=3}^t|\mathcal{L}(T_i, k_i)|=:f_1(G_1).$$
By the definition of $m'$,
$$m'=\max \left\{f_1(R): R \in {[n]\choose k_1} \ \text{and} \  I'_1\prec R\prec I''_1\right\}.$$

Denote
$$\mathcal{F}=\{F\in {[n]\choose k_1}: I'_1 \prec F \prec I''_1\}.$$
Then $m'=f_1(\mathcal{F})$ (recall that $f_1(\mathcal{F})=\max \{f_1(F): F\in \mathcal{F}\}$).
For each $j\in [k_1]$, denote
$$\mathcal{F}(j)=\{F\setminus [n-j+1, n]: F\in \mathcal{F}, [n-j+1, n]\subseteq F\}.$$
 To prove Claim \ref{clm2.21} is equivalent to prove  the following claim.

\begin{claim}\label{claim4.8}
$f_1(\mathcal{F})=\max \{f_1(I'_1), f_1(I''_1)\}$, and for any $F\in \mathcal{F}$ with $I'_1\precneqq F \precneqq I''_1$, we have $f_1(F)<f_1(\mathcal{F})$.
\end{claim}

\begin{proof}
By the definitions of $I'_1$ and $I''_1$,  we can see that $\ell(I'_1)=\ell(I''_1)=:k>0$, $\mathcal{F}(k)=\{I'_1\setminus [n-\ell(I'_1)+1, n],\, I''_1\setminus [n-\ell(I''_1)+1, n]\}$ and for each $F\in \mathcal{F}$, $\ell(F)\leq k$.
Denote
\begin{align*}
A_0&=\{x_1, \dots, x_{i}, x_{i+1}, x_{i+1}+1, \dots, x_{i+1}+k_1-i-1\},\\
A_1&=\{x_1, \dots, x_{i}, x_{i+1}, x_{i+1}+1,\dots, x_{i+1}+k_1-i-2\}\cup \{n\},\\
&\vdots\\
A_{k-1}&=\{x_1, \dots, x_{i}, x_{i+1}, x_{i+1}+1\}\cup [n-k_1+i+3, n].
\end{align*}

Then $A_0$ is the $k_1$-set with $F'_1<A_0$.
Applying Lemma \ref{clm28} repeatedly, we obtain
\begin{align*}
f_1(\mathcal{F})&=\max \{f_1(A_0), f_1(\mathcal{F}(1))\},\\
f_1(\mathcal{F}(1))&=\max \{f_1(A_1), f_1(\mathcal{F}(2))\},\\
&\vdots\\
f_1(\mathcal{F}(k-1))&=\max \{f_1(A_{k-1}), f_1(\mathcal{F}(k))\},\\
f_1(\mathcal{F}(k))&=\max \{f_1(F'_1), f_1(F''_1)\}.
\end{align*}
Thus
$$f_1(\mathcal{F})=\max \{f_1(A_0), f_1(A_1), \dots, f_1(A_{k-1}), f_1(I'_1), f_1(I''_1)\}.$$
Since $k_1>k_2$ and $n\geq k_2+k_t$,  $n> k_1+k_t$.
Then by Lemma \ref{clm28} again, we can see that for any $F\in \mathcal{F}\setminus\{A_0, A_1, \dots, A_{k-1}, I'_1, I''_1\}$, we have $f_1(F)<f_1(\mathcal{F})$.
Combing with the  following Claims \ref{lemma4.9} and \ref{lemma4.10}, we will get
Claim \ref{claim4.8}.

\begin{claim}\label{lemma4.9}
$\max \{f_1(A_0), f_1(A_1), \dots, f_1(A_{k-1}), f_1(I''_1)\}=\max \{f_1(A_0), f_1(I''_1)\}$.
\end{claim}

\begin{proof}
If $k=1$, then we are fine. We may assume that $k\geq 2$.
Let $I''_1=A_k$.
To prove Claim \ref{lemma4.9}, it is sufficient to prove that for any $j\in [0, k-2 ]$, if $f_1(A_j)\leq f_1(A_{j+1})$, then $f_1(A_{j+1})<f_1(A_{j+2})$.
Let $j\in [0, k-2]$.
Clearly, $\ell(A_{j+1})=\ell(A_j)+1\geq 1$ and $\ell(A_{j+2})=\ell(A_{j+1})+1$.
Let $A'_j$ and $A'_{j+1}$ be the $k_1$-sets such that $A_j<A'_j$ and $A_{j+1}<A'_{j+1}$.
Then $A'_j\overset{\ell(A_{j+1})}{\longrightarrow}A_{j+1}$.
Let $J$ be the $k_1$-set such that  $A'_{j+1}\overset{\ell(A_{i+1})}{\longrightarrow}J$. Then $A'_j$, $A_{j+1}$ are $\ell(A_{i+1})$-sequential and $A'_{j+1}$, $J$ are $\ell(A_{i+1})$-sequential. Clearly,
 $\max A'_j=\max A'_{j+1}$ and $\max A_{j+1}=\max J=n$. Applying Lemma \ref{coro22}, we have
\begin{align}\label{e26+}
\text{$\alpha(A'_j, A_{j+1})=\alpha(A'_{j+1}, J)$ and $\beta(A'_j, A_{j+1})=\beta(A'_{j+1}, J)$. }
\end{align}
If $\ell(A'_j)\geq 1$, then $A_j<A'_j=A_{j+1}<A_{j+2}=A'_{j+1}$.
In this case, $\alpha(A_{j+1}, A_{j+2})=1$.
By  Proposition \ref{clm23} and $n>k_1+k_t$ (since $k_1>k_2$ and $n\geq k_2+k_t$),
we have $\beta(A_{j+1}, A_{j+2})=0$.
So  $f(A_{j+1})< f(A_{j+2})$, as desired.
Next we may assume that $\ell(A'_j)=0$.
Clearly, $\alpha(A_j, A'_j)=\alpha(A_{j+1}, A'_{j+1})=1$. By Proposition \ref{clm23} and $\max A'_j=\max A'_{j+1}$, we have $\beta(A_j, A'_j)=\beta(A_{j+1}, A'_{j+1})$.
Combining with (\ref{e26+}), we get
\begin{align*}
\text{$\alpha(A_j, A_{j+1})=\alpha(A_{j+1}, J)$ and $\beta(A_j, A_{j+1})=\beta(A_{j+1}, J)$. }
\end{align*}
Thus $f(J)\geq f(A_{j+1})$ since $f(A_j)\leq f(A_{j+1})$.
Note that $\ell(A_{j+1})=\ell(J)$ and $A_{j+1}\setminus A_{j+1}^{\rm t}\in \mathcal{R}_1(\ell(A_{j+1}))$, so $J\setminus J^{\rm t}\in \mathcal{R}_1(\ell(A_{j+1}))$ and $A_{j+1}\setminus A_{j+1}^{\rm t} \overset{1}{\prec}J\setminus J^{\rm t}$. Since $\ell(A_{j+2})=\ell(A_{j+1})+1$, $A_{j+2}\setminus [n-\ell(A_{j+1})+1, n]\in \mathcal{R}_1(\ell(A_{j+1}))$. And in $\mathcal{R}_1(\ell(A_{j+1}))$,  we have
\[
A_{j+1}\setminus A_{j+1}^{\rm t} \overset{1}{\prec}J\setminus J^{\rm t}\overset{1}{\longrightarrow}A_{j+2}\setminus [n-\ell(A_{j+1})+1, n].
\]
Recall that $n>k_1+k_t$.
By Lemma \ref{clm28} and $f(J)\geq f(A_{j+1})$, we obtain $f(A_{j+2})>f(J)\geq f(A_{j+1})$, as required.
\end{proof}

\begin{claim}\label{lemma4.10}
If $f_1(A_0)\geq f_1(I'_1)$, then $f_1(A_0)< f_1(I''_1)$.
\end{claim}

\begin{proof}

Note that $I'_1<A_0$ in $\mathcal{R}_1$.
Then $\alpha(I'_1, A_0)=1$. Since $f_1(A_0)\geq f_1(I'_1)$, $\alpha(I'_1, A_0)\geq \beta(I'_1, A_0)$, so $\beta(I'_1, A_0)\leq 1$. Consider the family $\mathcal{F'}:=\{F: |F|=k_1, A_0\prec F\prec I''_1\}$. We can see that for each $F\in \mathcal{F'}$, $\max F\geq \max A_0$.
By Proposition \ref{clm23} and $n>k_1+k_t$ (since $k_1>k_2$ and $n\geq k_2+k_t$), for each two $k_1$-sets $F$ and $G$ with $A_0 \prec F<G\prec I''_1$, we have that $\beta(F, G)<\beta(I'_1, A_0)\leq 1$ and $\alpha(F, G)=1$.
Thus, $f_1(I''_1)>f_1(A_0)$. This completes the proof of Claim \ref{lemma4.10}
\end{proof}

Since we have shown Claims \ref{lemma4.9} and \ref{lemma4.10},  the proof of Claim \ref{claim4.8} is complete.
\end{proof}
Since the proof of Claim \ref{claim4.8} is complete,  Claim \ref{clm2.21} follows as noted before.
\end{proof}
Since the proof of Claim \ref{clm2.21} is complete, the  proof of Lemma \ref{l2.18} is complete.

\section{Verify Unimodality: Proof of Lemma \ref{c2.27}}\label{sec6}
In this section, we are going to prove a more general result Lemma \ref{lemma star} which will be applied for the most general case $n\geq k_1+k_t$ in \cite{mix2}. Lemma \ref{c2.27} will follow from Lemma \ref{lemma star}  immediately. Before  stating the result, we need to make some preparations.

\begin{definition}\label{def5.1}
Suppose that $t\geq 2$ is a positive integer, $s\in [t-1]$, $k_1\geq k_2\geq\dots\geq k_t$ and $k_1+k_{s+1}\leq n <k_{s-1}+k_s$. We define $\mathcal{R}_i$ for every $i\in [t]$ as follows.
For $i\in [s]$, let
\[
\mathcal{R}_i=\{R\in {[n]\choose k_i}: \{1, n-k_i+2, \dots, n\} \prec R \prec \{k_t, n-k_i+2, \dots, n\}\}.
\]
For $i\in [s+1, t]$, let
\[
\mathcal{R}_i=\{ R\in {[n]\choose k_i}: [k_t]\subseteq R \}.
\]
\end{definition}

In Section \ref{sec3}, we defined notations $\mathcal{R}_i$, $1\leq i \leq t$. Throughout the paper except in this section, $\mathcal{R}_i$ follows from the definitions in Section \ref{sec3}. $\mathcal{R}_i$ in this section follows from the above definition. When $s=2$, they are consistent. 

\begin{definition}
Let $R$ and $T$ be two subsets of $[n]$ with different sizes. We write $R\overset{\sim}{<} T$ if $R\prec T$ and there is no other set $R'$ such that  $|R'|=|R|$ and $R\precneqq R' \prec T$.
\end{definition}

By the definition of parity,  we have the following simple remark.
\begin{remark}
For any $R\in \mathcal{R}_1$ and $i\in [2, s]$, $R$ has a $k_i$-parity if and only if $|R\setminus R^{\rm t}|\leq k_i$.
\end{remark}
From the above observation, for any $R_1\in \mathcal{R}_1$ and $i\in [2, s]$, we define the {\it corresponding $k_i$-set} of $R_1$ as follows.
\begin{definition}\label{R_2}
Let $R_1\in \mathcal{R}_1$ and $i\in [2, s]$. If $R$ has a $k_i$-parity, then let $R_i$ be the $k_i$-parity of $R_1$; otherwise, let $R_i$ be the $k_i$-set such that $R_i\overset{\sim}{<} R_1$. We call $R_i$ the {\it corresponding $k_i$-set} of $R_1$.
\end{definition}

\begin{definition}\label{deff}
Let $R_1\in \mathcal{R}_1$. For each $i\in [2, s]$, let $R_i$ be  the corresponding $k_i$-set of $R_1$  as in Definition \ref{R_2} and for each $i\in [s+1, t]$, let $R_i$ be the $k_i$-partner of $R_1$. We denote
$$f(R_1)=\sum_{i=1}^t |\mathcal{L}(R_i, k_t)|.$$
\end{definition}

\begin{definition}\label{def6.6}
For each $j\in [k_1-1]$, let 
\begin{align*}
\mathcal{R}_{1, j}&=\{R\in \mathcal{R}_1: [n-j+1, n]\subseteq R\},\\
\mathcal{R}_1(j)&=\{R\setminus [n-j+1, n]: R\in \mathcal{R}_{1, j}\}. 
\end{align*}
For any $j\in [k_1-1]$ and any $R\in \mathcal{R}_{1, j}$, we define $f(R\setminus [n-j+1, n])=f(R)$. For example, $f(\{1\})=f(\{1, n-k_1+2, \dots, n\})$.
\end{definition}

\begin{definition}\label{ab}
Let $R_1, R'_1\in \mathcal{R}_1$ with $R_1\prec R'_1$ and for any $i\in [s+1, t]$, $R_i, R'_i$ be the $k_i$-partners of $R_1, R'_1$ respectively, and for each $i\in [2, s]$, let $R_i, R'_i$ be the corresponding $k_i$-sets of $R_1, R'_1$ respectively  as in Definition \ref{R_2}.  We define the following functions.

For $i\in [s]$, let
\begin{align*}
\alpha_i(R_1, R'_1)&=|\mathcal{L}(R'_i, k_1)|-|\mathcal{L}(R_i, k_1)|,\\
\gamma(R_1, R'_1)&=\sum_{i=1}^s\alpha_i(R_1, R'_1),\\
\delta(R_1, R'_1)&=\sum_{i=s+1}^t\big(|\mathcal{L}(R_i, k_i)|-|\mathcal{L}(R'_i, k_i)|\big).
\end{align*}
\end{definition}

\begin{definition}
For any $j\in [k_1-1]$ and any $R_1, R'_1 \in \mathcal{R}_{1, j}$ with $R_1\prec R'_1$, we define $\alpha_i(R_1\setminus [n-j+1, n], R'_1\setminus [n-j+1, n])=\alpha_i(R_1, R'_1)$ for each $i\in [s]$, $\gamma(R_1\setminus [n-j+1, n], R'_1\setminus [n-j+1, n])=\gamma(R_1, R'_1)$ and $\delta(R_1\setminus [n-j+1, n], R'_1\setminus [n-j+1, n])=\delta(R_1, R'_1)$.
\end{definition}

From the above definition, we have 
\[
f(R'_1)-f(R_1)=\gamma(R_1, R'_1)-\delta(R_1, R'_1).
\]
\begin{remark}\label{add}
Suppose that $A_1, B_1, C_1\in \mathcal{R}_1$ with $A_1\prec B_1\prec C_1$. Then
for any $i\in [s]$,
$
\alpha_i(A_1, C_1)=\alpha_i(A_1, B_1)+\alpha_i(B_1, C_1)$, and
$\gamma(A_1, C_1)=\gamma(A_1, B_1)+\gamma(B_1, C_1)$,
 $\delta(A_1, C_1)=\delta(A_1, B_1)+\delta(B_1, C_1)$.
\end{remark}

We will prove the following more general lemma, which implies Lemma \ref{c2.27}.

\begin{lemma}\label{lemma star}
Let $k\in [0, k_1-1]$ and $F_1, G_1, H_1 \in \mathcal{R}_1(k)$ with $F_1\overset{c}{\prec} G_1 \overset{c}{\prec} H_1$ for some $c\in [k_1]$. If $f(F_1)\leq f(G_1)$, then $f(G_1)<f(H_1)$.
\end{lemma}

Let us explain why Lemma \ref{lemma star} implies Lemma \ref{c2.27}. Take $s=2$ (see Definition \ref{def5.1}). Take $F_2, G_2, H_2 \in \mathcal{F}_{2, 3}(j)$ satisfying the condition $F_2 \overset{c}{\prec}G_2\overset{c}{\prec}H_2$ (see Lemma \ref{c2.27}). Let $F'_2=F_2\cup [n-j+1, n], G'_2=G_2\cup [n-j+1, n]$ and  $H'_2=H_2\cup [n-j+1, n]$. Then $F'_2, G'_2, H'_2 \in \mathcal{F}_{2, 3}$ and $g(F_2)=g(F'_2)$, $g(G_2)=g(G'_2)$ and $g(H_2)=g(H'_2)$ (see Definition \ref{def3.11}). Let $F'_1, G'_1, H'_1$ be the $k_1$-parities of $F'_2, G'_2, H'_2$ respectively. (Proposition \ref{prop2.9}) guarantees the $k_1$-parities of $F'_2, G'_2, H'_2$ exist.)
Take $k=j+k_1-k_2$ in Lemma \ref{lemma star}.
Let $F_1=F'_1\setminus [n-k+1, n]$, $G_1=G'_1\setminus [n-k+1, n]$ and $H_1=H'_1\setminus [n-k+1, n]$. Then $F_1, G_1, H_1 \in \mathcal{R}_1(k)$ with $F_1\overset{c}{\prec}G_1\overset{c}{\prec}H_1$ and $f(F_1)=g(F'_1)$, $f(G_1)=f(G'_1)$ and $f(H_1)=f(H'_1)$ (see Definition \ref{def6.6}).
By Fact \ref{fact2.7}, for each $i\in [3, t]$, $F'_1$ and $F'_2$ have the same $k_i$-partner,
$G'_1$ and $G'_2$ have the same $k_i$-partner and $H'_1$ and $H'_2$ have the same $k_i$-partner.
Therefore, $g(F_2)=g(F'_2)=f(F'_1)=f(F_1)$, $g(G_2)=g(G'_2)=f(G'_1)=f(G_1)$ and $g(H_2)=g(H'_2)=f(H'_1)=f(H_1)$. Now applying Lemma \ref{lemma star}, we have that if $f(F_1)=f(F'_1)=g(F_2)\leq g(G_2)=f(G'_1)=f(G_1)$, then $g(G_2)=f(G'_1)=f(G_1)<f(H_1)=f(H'_1)=g(H_2)$.

Before proving Lemma \ref{lemma star}, we need to make some preparations.
Denote
\begin{equation}\label{s'}
s'=\{i: i\in [s], k_i=k_1\}.
\end{equation}

The following remark is  useful.
\begin{remark}\label{remark5.9}
When $k_1=k_2=\dots=k_s$,  the authors have proved the truth of Theorem \ref{main1} in \cite{HP+}, see Corollary 1.12 in \cite{HP+} (taking $k=k_1=\dots=k_s$ in Corollary 1.12). So we may assume that $s'<s$. So $n>k_{s'}+k_{s+1}$.
\end{remark}

\begin{claim}\label{claim1}
Let  $R_1, R'_1\in \mathcal{R}_1$ with $R_1<R'_1$ and $R=R'_1\setminus {R'_1}^{\rm t}$
Then for each $i\in [s]$, we have
\[
\alpha_i(R_1, R'_1)={\ell(R'_1)\choose k_i-|R|}.
\]
In particular, $\alpha_i(R_1, R'_1)=0$ if and only if $\ell(R'_1)< k_1-k_i$. Furthermore, if $\ell(R'_1)=0$, then $\gamma(R_1, R'_1)=s'$.
\end{claim}

\begin{proof}
Clearly, for $i\in [s']$, $\alpha_i(R_1, R'_1)=1={\ell(R'_1)\choose k_i-|R|}$.
We next consider for $i\in[s'+1, s]$. In this case, $k_i<k_1$.
Let $R_i$ and $R'_i$ be the corresponding $k_i$-sets of $R_1$ and $R'_1$ respectively  as in Definition \ref{R_2}.  Then $R_i\prec R'_i$, moreover $\alpha_i(R_1, R'_1) \geq 0$ and  $\alpha_i(R_1, R'_1)=0$ if and only if $R_i=R'_i$ clearly. $R_i=R'_i$ implies $R'_i\overset{\sim}{<}R'_1$, furthermore, $\ell(R'_1)< k_1-k_i$.
Then  we can see that if $\ell(R'_1)=0$, then $\alpha_i(R_1, R'_1)=0={\ell(R'_1)\choose k_i-|R|}$. Therefore, $\gamma(R_1, R'_1)=s'$,
as required.
We next assume that $R'_i$ is the $k_i$-parirty of $R'_1$. So $\ell(R'_1)\geq 1$. In this case,
since $R_1<R'_1$, $\ell(R_1)=\ell(R'_1)-1$. If $R_i\overset{\sim}{<}R_1$, then
$\ell(R'_1)= k_i-|R|$ and
$\alpha_i(R_1, R'_1)=1={\ell(R'_1)\choose k_i-|R|}$, as required.
Last, we  assume that $R_i$ is the $k_i$-parirty of $R_1$. So $\ell(R_1)\geq 1$.
Let $k_1-k_i=k$.
In this case we have
\begin{align*}
R'_1&=R\cup [n-\ell(R'_1)+1, n];\\
R_1&=R\cup \{n-\ell(R'_1)\} \cup [n-\ell(R'_1)+2, n];\\
R_i&=R\cup \{n-\ell(R'_1)\} \text { if $\ell(R_i)=0$,}\\
R_i&=R\cup \{n-\ell(R'_1)\}\cup [n-\ell(R'_1)+2+k, n] \text{ if $\ell(R_i)\geq 1$};\\
R'_i&=R\cup [n-\ell(R'_1)+1+k, n].
\end{align*}
Then
\[
\alpha_i(R_1, R'_1)=|\mathcal{L}(R'_i, k_i)|-|\mathcal{L}(R_i, k_i)|={\ell(R'_1)\choose k_i-|R|},
\]
as required.
\end{proof}

From Definitions \ref{ab} and \ref{ab-}, we have the following observation.
\begin{remark}\label{remark4.14}
For any $R_1, R_1' \in \mathcal{R}_1$ with $R_1\prec R'_1$, we have that
$
\alpha_i(R_1, R'_1)=\alpha(R_1, R'_1)$ holds for each $i\in [s']$, and
$\delta(R_1, R'_1)=\beta(R_1, R'_1)$.
\end{remark}

\begin{claim}\label{equal}
Let $j\in [0, k_1-1]$, $d\in [k_1-j]$ and $A_1, B_1, C_1, D_1\in \mathcal{R}_1(j)$.
Suppose that $A_1, B_1$ are $d$-sequential and  $C_1, D_1$ are $d$-sequential with $\max A_1=\max C_1$ and $\max B_1=\max D_1$. Then $\gamma(A_1, B_1)=\gamma(C_1, D_1)$ and
$\delta(A_1, B_1)=\delta(C_1, D_1)$.
In particular, if $A_1\overset{d}{\longrightarrow} B_1$, $C_1\overset{d}{\longrightarrow} D_1$ and $\max A_1=\max C_1$, then $\gamma(A_1, B_1)=\gamma(C_1, D_1)$ and
$\delta(A_1, B_1)=\delta(C_1, D_1)$.
\end{claim}

\begin{proof}
 By the definitions of $A_1, B_1, C_1, D_1$, from Lemma \ref{clm27-} and Remark \ref{remark4.14}, we have
 $\delta(A_1, B_1)=\delta(C_1, D_1)$ and $\alpha_i(A_1, B_1)=\alpha_i(C_1, D_1)$ holds for each $i\in [s']$.  Next, we aim to show that for each $i\in [s'+1, s]$, $\alpha_i(A_1, B_1)=\alpha_i(C_1, D_1)$. Let $i\in [s'+1, s]$.
Denot
\begin{align*}
\mathcal{A}&=\{R: A_1\cup [n-j+1, n]\prec R\prec B_1\cup [n-j+1, n] \text{ and } |R|=k_1\},\\
\mathcal{B}&=\{T: C_1\cup [n-j+1, n]\prec T\prec D_1\cup [n-j+1, n]\text{ and } |T|=k_1\}.
\end{align*}
Since $\alpha_1(A_1, B_1)=\alpha_1(C_1, D_1)$, $|\mathcal{A}|=|\mathcal{B}|=:h$.
Let $\mathcal{A}=\{R_1, R_2, \dots, R_h\}$ and $\mathcal{B}=\{T_1, T_2, \dots, T_h\}$, where $R_1 \prec R_2 \prec \dots \prec R_h$ and $T_1 \prec T_2 \prec \dots \prec T_h$.
For any $j\in [h]$, we have $\ell(R_j)=\ell(T_j)$ and $|R_j\setminus R_j^{\rm t}|=|T_j\setminus T_j^{\rm t}|$. Thus, by Claim \ref{claim1}, for any $j\in [h-1]$, $\alpha_i(R_j, R_{j+1})=\alpha_i(T_j, T_{j+1})$. Furthermore, by Remark \ref{add}, we conclude that $\alpha_i(A_1, B_1)=\alpha_i(C_1, D_1)$.
\end{proof}

\begin{claim}\label{claim4}
Let $A_1, B_1, C_1\in \mathcal{R}_1$ and $a$ be an integer. Suppose  $A_1\setminus A_1^{\rm t}=A\cup \{a, a+1\}$, $B_1=A\cup \{a\}\cup [n-\ell(B_1)+1, n]$ and $C_1=A\cup \{a+1\}\cup [n-\ell(B_1)+1, n]$ for some  $A$.
Then $\delta(A_1, B_1)=\delta(B_1, C_1)$.
If $a+1\leq n-\ell(A_1)-2$, then $\gamma(A_1, B_1)=\gamma(B_1, C_1)$. If $a+1= n-\ell(A_1)-1$, then $\gamma(A_1, B_1)\leq \gamma(B_1, C_1)$, equality holds if and only if $C_1$ does not have $k_{s'+1}$-parity (recall that $s'$ is the integer such that $k_1=\dots=k_{s'}>k_{s'+1}$).
\end{claim}

\begin{proof}
Let $A'_1$ and $B'_1$ be the $k_1$-sets such that $A_1<A'_1$ and $B_1<B'_1$.
Then $\max A'_1=\max B'_1$, and $A'_1$, $B_1$ are $(\ell(A_1)+1)$-sequential,
$B'_1$, $C_1$ are $(\ell(A_1)+1)$-sequential. Note that $\max B_1=\max C_1=n$, then  applying Lemma \ref{coro22} and Reamrk \ref{remark4.14}, we have that $\alpha_i(A'_1, B_1)=\alpha_i(B'_1, C_1)$ holds for each $i\in [s']$ and $\delta(A'_1, B_1)=\delta(B'_1, C_1)$.
Clearly, $\alpha_1(A_1, A'_1)=\alpha_1(B_1, B'_1)=1$ holds for each $i\in [s']$ and using Proposition \ref{clm23}
and Remark \ref{remark4.14}, we have $\delta(A_1, A'_1)=\delta(B_1, B'_1)$. Therefore, by Remark \ref{add}, we have $\delta(A_1, B_1)=\delta(B_1, C_1)$. Clearly, for each $i\in [s']$,
\begin{equation}\label{eq13}
\alpha_i(A_1, B_1)=\alpha_i(A_1, A'_1)+\alpha_i(A'_1, B_1)=\alpha_i(B_1, B'_1)+\alpha_i(A'_1, C_1)=\alpha_i(B_1, C_1).
\end{equation}
Since $A_1\setminus A_1^{\rm t}=A\cup \{a, a+1\}$, $a+1\leq n-\ell(A_1)-1$.
By the definitions of $A_1$, $B_1$, and  $C_1$,
if $a+1\leq n-\ell(A_1)-2$, then
$\ell(B_1)=\ell(C_1)=\ell(A_1)+1$, and if $a+1= n-\ell(A_1)-1$, then
$\ell(B_1)=\ell(A_1)+1$ and $\ell(C_1)=\ell(B_1)+1$.
 Using Claim \ref{equal}, for each $j\in [s'+1, s]$, we have
 \begin{equation}\label{eq14}
 \alpha_j(A'_1, B_1)=\alpha_j(B'_1, C_1).
 \end{equation}
If the previous case happens, then $\ell(A'_1)=\ell(B'_1)=0$, so Claim \ref{claim1} gives $\alpha_i(A_1, A'_1)=\alpha_i(B_1, B'_1)=0$ for each $i\in [s'+1, s]$. Combing this with (\ref{eq13}), (\ref{eq14}) and Remark \ref{add}, we get
\begin{align*}
\gamma(A_1, B_1)&=\sum_{i=1}^s\alpha_i(A_1, B_1)\\
&=\sum_{i=1}^{s'}\alpha_i(A_1, B_1)+\sum_{i=s'+1}^s(\alpha_i(A_1, A'_1)+\alpha_i(A'_1, B_1))\\
&=\sum_{i=1}^{s'}\alpha_1(B_1, C_1)+\sum_{i=s'+1}^s(\alpha_i(B_1, B'_1)+\alpha_i(B'_1, C_1))\\
&=\gamma(B_1, C_1),
\end{align*}
as desired.
If the later case happens, then $A_1<B_1<C_1$, by Claim \ref{claim1}, since $\ell(B_1)=\ell(A_1)+1$ and $\ell(C_1)=\ell(B_1)+1$, then for each $j\in [s'+1, s]$, we have $\alpha_j(A_1, B_1)= \alpha_j(B_1, C_1)$ if $k_j<|C_1\setminus C_1^{\rm t}|$, i.e., $C_1$ dest not have $k_j$-parity; and $\alpha_j(A_1, B_1)<\alpha_j(B_1, C_1)$ if $k_j\geq |C_1\setminus C_1^{\rm t}|$, i.e., $C_1$ has $k_j$-parity. Note that $k_{s'+1}\geq \dots \geq k_s$, so if $C_1$ does not have $k_{s'+1}$-parity, then it does not have any $k_j$-parity for $j\in [s'+1, s]$. 
Therefore, $\sum_{j=1}^s \alpha_j(A_1, B_1)\leq \sum_{j=1}^s \alpha_j(B_1, C_1)$, and the equality holds if and only if $C_1$ has $k_{s'+1}$-parity, that is, $\gamma(A_1, B_1)\leq \gamma(B_1, C_1)$, and the equality holds if and only if $C_1$ does  not have $k_{s'+1}$-parity, as desired.
\end{proof}

We are going to prove Lemma \ref{lemma star} by induction on $k$.

\subsection{\bf{ The Base Case }}
We are going to show that Lemma \ref{lemma star} holds for $k=0$. Assume that $F_1, G_1, H_1\in \mathcal{R}_1$ with $F_1 \overset{c}{\prec}G_1\overset{c}{\prec}H_1$ for some $c\in [k_1]$ and $f(F_1)\leq f(G_1)$, we are going to show that $f(G_1)<f(H_1)$.
Since $F_1, G_1, H_1 \in \mathcal{R}_1$ with $F_1\overset{c}{\prec} G_1 \overset{c}{\prec} H_1$,
$\max F_1\leq n-2$, $\max G_1\leq n-1$ and $\max H_1\leq n$.
Let $G'_1$ and $H'_1$ be the $k_1$-sets such that $G'_1<G_1$ and $H'_1<H_1$.
we first show  the following observation.

\begin{claim}\label{claim2}
Let $F_m$ be the $k_1$-set such that $F_1\overset{c}{\longrightarrow}F_m$, where $m$ is the number of $k_1$-subsets $R$ satisfying $F_1\prec R \prec F_m$.
If $\delta(G'_1, G_1)<s'$, then $f(R)$, where $R\in \mathcal{F}$, increases on  $\mathcal{F}$, in particular, $f(G_1)<f(H_1)$.
\end{claim}

\begin{proof}
We may assume that
$\mathcal{F}=\{F_i: 1\leq i \leq m\}$ and
$F_1<F_2<\dots<F_m$. Then $\max F_2=\max G_1$. Since $F_1<F_2$ and $G'_1<G_1$, by Proposition \ref{clm23} and Remark \ref{remark4.14}, $\delta(F_1, F_2)=\delta(G'_1, G_1)<s'$ (assumption). Since $\max F_i\geq F_2$ for any $i\in [2, m]$, then by Proposition \ref{clm23} and Remark \ref{remark5.9}, $\delta(F_i, F_{i+1})<s'$ holds for all $i\in [m-1]$. On the other hand, $\gamma(F_i, F_{i+1})\geq s'$ for all $i\in [m-1]$, thus, $f(R)$, where $R\in \mathcal{F}$, increases on  $\mathcal{F}$. Clearly, $G_1, H_1\in \mathcal{F}$ and $G_1\precneqq H_1$, therefore, $f(G_1)<f(H_1)$, as required.
\end{proof}

By Claim \ref{claim2}, next we may assume that
\begin{align}\label{equality15}
\delta(G'_1, G_1)\geq s'.
\end{align}
Note that
\begin{align}\label{equality16}
f(G_1)&=f(G'_1)+\gamma(G'_1, G_1)-\delta(G'_1, G_1),\\ \label{equality17}
f(H_1)&=f(H'_1)+\gamma(H'_1, H_1)-\delta(H'_1, H_1).
\end{align}
Since $\ell(G_1)=0$, by Claim \ref{claim1}, for each $i\in [s'+1, s]$, we have
\begin{equation}\label{e29}
\alpha_i(G'_1, G_1)=0\leq \alpha_i(H'_1, H_1).
\end{equation}
Clearly, for $i\in [s']$, $\alpha_i(G'_1, G_1)=\alpha_i(H'_1, H_1)=1$.  So $\gamma(G'_1, G_1)\leq\gamma(H'_1, H_1)$.
Note that $\max G_1<\max H_1$, combining Proposition \ref{clm23}, Remark \ref{remark5.9}, Remark \ref{remark4.14} and (\ref{equality15}), we obtain
$\delta(G'_1, G_1)> \delta(H'_1, H_1)$.
Combining with equalities (\ref{equality16}) and (\ref{equality17}),  to show $f(G_1)<f(H_1)$, it is sufficient to show the following claim.

\begin{claim}\label{claim3}
$f(G'_1)\leq f(H'_1)$.
\end{claim}

\begin{proof}
If $c=1$, then $F_1=G'_1$ and $G_1=H'_1$. Since $f(F_1)\leq f(G_1)$, $f(G'_1)\leq f(H'_1)$, as desired. Thus, we have proved  for $c=1$.

Next, we consider  $c\geq 2$. In this case, $F_1\overset{c-1}{\longrightarrow}G'_1$ and $G_1\overset{c-1}{\longrightarrow}H'_1$. Let $F'_1$ be the $k_1$-set such that $F_1\overset{c-1}{\prec}F'_1$. Therefore, $\max F'_1=\max G_1$ and $F'_1\overset{c-1}{\longrightarrow}G'_1$.
Note that $\max G'_1=\max H'_1=n$.
By Lemma \ref{coro22} and Remark \ref{remark4.14}, we have
$\alpha_1(F'_1, G'_1)=\alpha_1(G_1, H'_1)$ and $\delta(F'_1, G'_1)=\delta(G_1, H'_1)$.
Let $F''_1$ be the $k_1$-set such that $F''_1<F'_1$. Thus, $F''_1$, $G'_1$ and $H'_1$ satisfy the condition of Claim \ref{claim4}. So we conclude that $\gamma(F''_1, G'_1)\leq \gamma(G'_1, H'_1)$ and $\delta(F''_1, G'_1)=\delta(G'_1, H'_1)$. Note that
\begin{align*}
f(G'_1)=f(F''_1)+\gamma(F''_1, G'_1)-\delta(F''_1, G'_1),\\
f(H'_1)=f(G'_1)+\gamma(G'_1, H'_1)-\delta(G'_1, H'_1).
\end{align*}
Thus, to show $f(G'_1)\leq f(H'_1)$ is sufficient  to show the following claim.
\begin{claim}\label{claim5+}
$f(F''_1)\leq f(G'_1)$.
\end{claim}

\begin{proof}
Suppose on the contrary that $f(F''_1)>f(G'_1)$. Since
$\alpha_i(G'_1, G_1)=1$ holds for each $i\in [s']$ and
$\alpha_j(G'_1, G_1)=0$ holds for each $j\in [s'+1, s]$ (see (\ref{e29})), $\gamma(G'_1, G_1)=s'$. Since $\delta(G'_1, G_1)\geq s'$ and
$f(G_1)=f(G'_1)+\gamma(G'_1, G_1)-\delta(G'_1, G_1)$, $f(G'_1)\geq f(G_1)$. Therefore, $f(G'_1)\geq f(F_1)$ since $f(G_1)\geq f(F_1)$. Hence, $f(F''_1)>f(F_1)$ since $f(F''_1)>f(G'_1)$. This implies $F_1\ne F''_1$. Therefore, $c\geq 3$ and $F_1\overset{c-2}{\longrightarrow}F''_1$. Since $F_1\overset{c}{\prec}G_1$, we may assume that there exists some $(k_1-c)$-set $F$ and some integer $x$ such that $F_1=F\cup [x+1, x+c]$ and $\max F\leq x$ if $F\ne \emptyset$.
Then $F''_1=F\cup \{x+1, x+2\}\cup [n-c+3, n]$.

The forthcoming claim will be used to make a contradiction to $f(F''_1)>f(G'_1)$, thereby ending the proof of Claim \ref{claim5+}. We will explain this after stating the claim.

\begin{claim}\label{claim6}
Let $d$ be a positive integer, $B_1=B\cup [y+1, y+d]$ for some set $B$ with $\max B\leq y$ and $y+d<n$. Suppose $i\in [d]$ and $C_1$ is the $k_1$-set such that  $B_1\overset{i}{\longrightarrow}C_1$. Let $p=n-y-d$. We define $k_1$-sets $D_1, D_2, \dots, D_p$ as follows.
If $i=1$, then let $D_1=B_1$ and $D_j=B\cup [y+1, y+d-1] \cup \{y+d+j-1\}$ for each $j\in [2, p]$, when $d=1$,  we regard $[y+1, y+d-1]$ as an empty set. If $i\geq 2$, then let
$D_j=B\cup [y+1, y+d-i] \cup \{y+d+j-i\}\cup [n-i+2, n]$ for each $j\in [ p]$,  when $d=i$,  we regard $[y+1, y+d-i]$ as an empty set.
If $f(B_1)\leq f(C_1)$, then for each $j\in [p]$, $f(D_j)\leq f(C_1)$.
\end{claim}

To finish the proof of Claim \ref{claim5+}, we only need to prove Claim \ref{claim6}. Assume that Claim \ref{claim6} holds, taking $B_1=F_1$, $y=x$, $d=c$ and $i=c-1$ in Claim \ref{claim6}, we can see that $C_1=G'_1$ and $D_1=F''_1$. Since $f(F_1)\leq f(G'_1)$, then Claim \ref{claim6} gives $f(F''_1)\leq f(G'_1)$, which is a contradiction to the assumption that $f(F''_1)> f(G'_1)$.
This completes the proof of Claim \ref{claim5+} by assuming the truth of Claim \ref{claim6}.
\end{proof}

\begin{proof}[Proof for Claim \ref{claim6}]
We are going to prove Claim \ref{claim6} by induction on $i$.

We first consider the case $i=1$. In this case, $B_1=D_1\overset{1}{\prec}D_2\overset{1}{\prec}\dots\overset{1}{\prec}D_p\overset{1}{\prec}C_1$.
If $p=1$, then we have nothing to say. We may assume $p\geq 2$.
If  $f(D_2)> f(C_1)$, then $f(D_1)< f(D_2)$ since $f(B_1)=f(D_1)\leq f(C_1)$. Note that we have already proved Lemma \ref{lemma star} when $c=1$. By taking $c=1$ in Lemma \ref{lemma star}, we have $f(D_2)<f(D_3)<\dots<f(D_p)<f(C_1)$, a contradiction. Using the same argument, we can see that for each $j\in [2, p]$, $f(D_j)<f(C_1)$, as required.

We next consider the case $i\geq 2$ and assume that Claim \ref{claim6} holds for $i-1$.
If $f(D_1)\geq f(C_1)$, then $f(D_1)\geq f(B_1)$ since $f(B_1)\leq f(C_1)$. Note that $B_1\overset{i-1}{\longrightarrow}D_1$. By replacing $C_1$ with $D_1$ and $i$ with $i-1$ in Claim \ref{claim6}, we define $E_1, E_2, \dots, E_p$ as definitions of $D_1, D_2, \dots, D_p$. Then by induction hypothesis, we obtain
\begin{equation}\label{eq15}
\text{ $f(E_j)\leq f(D_1)$ holds for each $j\in [p]$.}
\end{equation}

For convenience, we denote $C_1=D_{p+1}$. We have the following claim.

\begin{claim}\label{claim7}
 For each $j\in [p]$, $\delta(E_j, D_1)=\delta(D_j, D_{j+1})$.
 For each $j\in [p-1]$, $\gamma(E_j, D_1)=\gamma(D_j, D_{j+1})$ and
 $\gamma(E_p, D_1)\leq \gamma(D_p, D_{p+1})$.
\end{claim}

\begin{proof}
Note that $E_p<D_1$, $D_p<D_{p+1}$, $\ell(D_1)=\ell(D_{p+1})-1$ and $|D_1\setminus D_1^{\rm t}|=|D_{p+1}\setminus D_{p+1}^{\rm t}|+1$. By Claim \ref{claim1}, for each $j\in [s'+1, s]$, $\alpha_j(E_p, D_1)\leq \alpha_j(D_p, D_{p+1})$. Trivially, for each $j\in [s']$, $\alpha_i(E_p, D_1)\leq \alpha_i(D_p, D_{p+1})=1$, therefore, $\gamma(E_p, D_1)\leq \gamma(D_p, D_{p+1})$, as required.

For each $j\in [2, p]$, let $F_j$ and $H_j$  be the $k_1$-sets such that $D_{j-1}<F_j$ and  $E_{j-1}<H_j$. Thus, for each $j\in [2, p]$,
$F_j \overset{i-1}{\longrightarrow}D_j$ and we have
$
B_1 \overset{i-1}{\prec} H_2 \overset{i-1}{\prec} H_3 \overset{i-1}{\prec}\dots \overset{i-1}{\prec} H_p \overset{i-1}{\prec} D_1
$
and
$
B_1 \overset{i}{\prec} F_2 \overset{i}{\prec} F_3 \overset{i}{\prec}\dots \overset{i}{\prec} F_p \overset{i}{\prec} C_1.
$
Then for each $j\in [2, p]$, we have $\max H_j=\max F_j$. By Claim \ref{equal},  $\gamma(H_j, D_1)=\gamma(F_j, D_j)$ and $\delta(H_j, D_1)=\delta(F_j, D_j)$.
Note that for each $j\in [2, p]$, $D_{j-1}<F_j$ and $E_{j-1}< H_j$, hence
Claim \ref{claim1} gives $\gamma(D_{j-1}, F_j)=\gamma(E_{j-1},  H_j)$, Proposition \ref{clm23} and Remark \ref{remark4.14} give $\delta(D_{j-1}, F_j)=\delta(E_{j-1},  H_j)$. So for each $j\in[2, p]$, we have 
\begin{align*}
\gamma(E_{j-1}, D_1)&=\gamma(E_{j-1}, H_j)+\gamma(H_j, D_1)\\
&=\gamma(D_{j-1}, F_j)+\gamma(F_j, D_j)\\
&=\gamma(D_{j-1}, D_j),
\end{align*}
and
\begin{align*}
\delta(E_{j-1}, D_1)&=\delta(E_{j-1}, H_j)+\delta(H_j, D_1)\\
&=\delta(D_{j-1}, F_j)+\delta(F_j, D_j)\\
&=\delta(D_{j-1}, D_j).
\end{align*}
Since $E_p<D_1$, $D_p<D_{p+1}$ and $\max D_1=\max D_{p+1}=\max C_1=n$. By
Proposition \ref{clm23} and Remark \ref{remark4.14}, $\delta(E_p, D_1)=\beta(E_p, D_1)=\beta(D_p, D_{p+1})=\delta(D_p, D_{p+1})$.
This completes the proof of Claim \ref{claim7}.
\end{proof}

Let us continue the proof of Claim \ref{claim6}.
Note that for each $j\in [p]$, $f(D_{j+1})=f(D_j)+\gamma(D_j, D_{j+1})-\delta(D_j, D_{j+1})$ and $f(D_{1})=f(E_j)+\gamma(E_j, D_{1})-\delta(E_j, D_{1})$.
By Claim \ref{claim7} and (\ref{eq15}), we conclude that
\[
f(D_1)\leq f(D_2)\leq \dots \leq f(D_{p+1})=f(C_1).
\]
This completes the proof of Claim \ref{claim6}.
\end{proof}
Since we have shown that Claim \ref{claim6} holds,  Claim \ref{claim3} holds.
\end{proof}
This completes the base case of  Lemma \ref{lemma star}. We next to consider the induction step.

\subsection{\bf{ The Induction Step }}
Recall that $\mathcal{R}_{1, k}=:\{R\in\mathbb{R}_1: [n-k+1, n]\subset R\},
\mathcal{R}_1(k)=:\{R\setminus [n-k+1, n]: R\in\mathcal{R}_{1, k}\}$ for $k\in [k_1-1]$.
The authors have shown the following result in \cite{HP}.

\begin{claim}\cite{HP}\label{0000-}
Let $j\in [0, k_1-j]$, $F_1<F'_1, G_1<G'_1$ in $\mathcal{R}_1(j)$, and  $\max F'_1=\max G'_1$. Then $\alpha(F_1, F'_1)=\alpha(G_1, G'_1)$ and $\beta(F_1, F'_1)=\beta(G_1, G'_1)$.
\end{claim}

We have the following claim.
\begin{claim}\label{0000}
Let $j\in [0, k_1-1]$, $F_1<F'_1, G_1<G'_1$ and $F_1 \prec F'_1 \prec G_1 \prec G'_1$ in $\mathcal{R}_1(j)$, and  $\max F'_1=\max G'_1$ with $\ell(F'_1)=\ell(G'_1)$ in $\mathcal{R}_1(j)$.
Then $\gamma(F_1, F'_1)=\gamma(G_1, G'_1)$ and $\delta(F_1, F'_1)=\delta(G_1, G'_1)$.
\end{claim}

\begin{proof}
If $j=0$, then we are fine (recall Claim \ref{claim1} and Proposition \ref{clm23}).
Assume $j\geq 1$.
Let $A_1=F_1\cup [n-j+1, n]$, $A'_1=F'_1\cup [n-j+1, n]$, $B_1=G_1\cup [n-j+1, n]$ and $B'_1=G'_1\cup [n-j+1, n]$.
 Let $A''_1$ and $B''_1$ be the  $k_1$-sets such that $A_1<A''_1$ and $B_1<B''_1$. Then by the definitions of $F_1, F'_1, G_1, G'_1$, we have $\max A''_1=\max B''_1$, $A''_1\overset{j}{\longrightarrow}A'_1$ and $B''_1\overset{j}{\longrightarrow}B'_1$. By Claim \ref{equal}, $\gamma(A''_1, A'_1)=\gamma(B''_1, B'_1)$ and $\delta(A''_1, A''_1)=\delta(B''_1, B'_1)$. By Claim \ref{claim1} and $\ell(F'_1)=\ell(G'_1)$ in $\mathcal{R}_1(j)$, $\gamma(A_1, A'_1)=\gamma(B_1, B''_1)$ and $\delta(A_1, A''_1)=\delta(B_1, B''_1)$. Thus, $\gamma(A_1, A'_1)=\gamma(B_1, B'_1)$ and $\delta(A_1, A'_1)=\delta(B_1, B'_1)$, that is $\gamma(F_1, F'_1)=\gamma(G_1, G'_1)$ and $\delta(F_1, F'_1)=\delta(G_1, G'_1)$,
 as desired.
\end{proof}


We are ready to give  the induction step of Lemma \ref{lemma star}.

\begin{proof}[Proof of Lemma \ref{lemma star}]
By induction on $k$. We have shown that it holds for $k=0$ in Section 6.1. Suppose it holds for $k\in[0, k_1-2]$, we are going to prove it holds for $k+1$. Let $F_1, G_1, H_1 \in \mathcal{R}_1(k+1)$ with $F_1\overset{c}{\prec}G_1\overset{c}{\prec}H_1$ and $f(G_1)\geq f(F_1)$, i.e., $\gamma(F_1, G_1)
\geq \delta(F_1, G_1)$. We are going to apply induction hypothesis to show $f(H_1)>f(G_1)$, i.e., $\gamma(G_1, H_1)>\delta(G_1, H_1)$.
Let $F'_1=F_1\cup \{\max F+1\}, G'_1=G_1\cup \{\max G_1+1\}$ and $H'_1=H_1\cup \{\max H_1+1\}$.
Then $F'_1, G'_1, H'_1\in \mathcal{R}_1(k)$. Moreover, $F'_1\overset{c+1}{\prec}G'_1\overset{c+1}{\prec}H'_1$ in  $\mathcal{R}_1(k)$.

Let $A, B, C, D, E$ be the $(k_1-k)$-sets satisfying $C<G'_1<D, E<H'_1, F'_1\overset{c}{\prec}A$ and $F'_1<B$ in $\mathcal{R}_1(k)$. Let $\widetilde{F}_1=F_1\sqcup\{n-k\}, \widetilde{G}_1=G_1\sqcup\{n-k\}$, $\widetilde{H}_1=H_1\sqcup\{n-k\}$. Then $\widetilde{F}_1, \widetilde{G}_1, \widetilde{H}_1\in \mathcal{R}_1(k)$.
We can see that
if $c\geq 2$, then
\begin{equation}\label{eq22}
F'_1<B\overset{1}{\longrightarrow} \widetilde{F}_1\precneqq A\overset{c}{\longrightarrow} C<G'_1<D\overset{1}{\longrightarrow}  \widetilde{G}_1\, \,\text{and}\,\,G'_1\overset{c}{\longrightarrow}E<H'_1.
\end{equation}
If $c=1$, then
\begin{equation}\label{eq22+}
F'_1<A=B\overset{1}{\longrightarrow} \widetilde{F}_1=C<G'_1<D\overset{1}{\longrightarrow}  \widetilde{G}_1\, \,\text{and}\,\,G'_1\overset{1}{\longrightarrow}E<H'_1.
\end{equation}

\begin{claim}\label{00000}
$\gamma(A, C)>\delta(A, C).$
\end{claim}

\begin{proof}
Suppose on the contrary that $\gamma(A, C)\leq\delta(A, C).$
We first consider the case $c\geq 2$.
By (\ref{eq22}),
\begin{align*}
&\gamma(\widetilde{F}_1, \widetilde{G}_1)=\gamma(\widetilde{F}_1, A)+\gamma(A, C)+\gamma(C, G'_1)+\gamma(G'_1, \widetilde{G}_1),\\
&\delta(\widetilde{F}_1, \widetilde{G}_1)=\delta(\widetilde{F}_1, A)+\delta(A, C)+\delta(C, G'_1)+\delta(G'_1, \widetilde{G}_1).
\end{align*}
Note that
$f(F_1)=f(\widetilde{F}_1)$ and $f(G_1)=f(\widetilde{G}_1)$, therefore,
$\gamma(F_1, G_1)\geq\delta(F_1, G_1)$ implies $\gamma(\widetilde{F}_1, \widetilde{G}_1)\geq\delta(\widetilde{F}_1, \widetilde{G}_1)$. Since $\gamma(A, C)\leq\delta(A, C)$, then
\begin{equation}\label{eq23}
\gamma(\widetilde{F}_1, A)+\gamma(C, G'_1)+\gamma(G'_1, \widetilde{G}_1)\geq\delta(\widetilde{F}_1, A)+\delta(C, G'_1)+\delta(G'_1, \widetilde{G}_1).
\end{equation}
Note that $\max B=\max G'_1$ with $\ell(B)=\ell(G'_1)=0$ in $\mathcal{R}_1(j)$ (i.e., $\max B=\max G'_1<n-k$). By Claim \ref{0000}, we have $\delta(F'_1, B)=\delta(C, G'_1)$ and $\gamma(F'_1, B)=\gamma(C, G'_1)$. Note that $B\overset{1}{\longrightarrow} \widetilde{F}_1, G'_1\overset{1}{\longrightarrow} \widetilde{G}_1$, $\max B=\max G'_1$ and $\max  \widetilde{F}_1=\max  \widetilde{G}_1 $, it follows from Claim \ref{equal} that
$\gamma(B, \widetilde{F}_1)=\gamma(G'_1, \widetilde{G}_1)$ and $\delta(B, \widetilde{F}_1)=\delta(G'_1, \widetilde{G}_1).$  Then
\begin{align*}
\gamma(\widetilde{F}_1, A)+\gamma(C, G'_1)+\gamma(G'_1, \widetilde{G}_1)
&=\gamma(\widetilde{F}_1, A)+\gamma(F'_1, B)+\gamma(B, \widetilde{F}_1)=\gamma(F'_1, A).
\end{align*}
Similarly, we have
$$\delta(\widetilde{F}_1, A)+\delta(C, G'_1)+\delta(G'_1, \widetilde{G}_1)=\delta(F'_1, A).$$
So inequality (\ref{eq23}) gives $\gamma(F'_1, A)\geq\delta(F'_1, A)$.
Note that $F'_1\overset{c}{\prec}A, A\overset{c}{\longrightarrow}C$ in $\mathcal{R}_1(k)$ for $c\in [k_i-k]$, by induction hypothesis, $\gamma(A, C)>\delta(A, C)$. A contradiction to our assumption.

We next consider the case $c=1$. 
By (\ref{eq22+}), we have
\begin{align*}
&\gamma(\widetilde{F}_1, \widetilde{G}_1)=\gamma(C, G'_1)+\gamma(G'_1, \widetilde{G}_1),\\
&\delta(\widetilde{F}_1, \widetilde{G}_1)=\delta(C, G'_1)+\delta(G'_1, \widetilde{G}_1).
\end{align*}
Note that $\gamma(F_1, G_1)\geq\delta(F_1, G_1)$ implies $\gamma(\widetilde{F}_1, \widetilde{G}_1)\geq\delta(\widetilde{F}_1, \widetilde{G}_1)$ and
\begin{equation}\label{eq23+}
\gamma(C, G'_1)+\gamma(G'_1, \widetilde{G}_1)\geq\delta(C, G'_1)+\delta(G'_1, \widetilde{G}_1).
\end{equation}
Note that $\max B=\max G'_1$ and $\ell(B)=\ell(G'_1)=0$ in $\mathcal{R}_1(k)$. By Claim \ref{0000}, we have $\delta(F'_1, B)=\delta(C, G'_1)$ and $\gamma(F'_1, B)=\gamma(C, G'_1)$. Note that $B\overset{1}{\longrightarrow} \widetilde{F}_1, G'_1\overset{1}{\longrightarrow} \widetilde{G}_1$, $\max B=\max G'_1$ and $\max  \widetilde{F}_1=\max  \widetilde{G}_1 $, it follows from Claim \ref{equal} that
$\gamma(B, \widetilde{F}_1)=\gamma(G'_1, \widetilde{G}_1)$ and $\delta(B, \widetilde{F}_1)=\delta(G'_1, \widetilde{G}_1).$  Then
\begin{align*}
\gamma(C, G'_1)+\gamma(G'_1, \widetilde{G}_1)
&=\gamma(F'_1, B)+\gamma(B, \widetilde{F}_1)=\gamma(F'_1, \widetilde{F}_1).
\end{align*}

Similarly, we have
$$\delta(C, G'_1)+\delta(G'_1, \widetilde{G}_1)=\delta(F'_1, \widetilde{F}_1).$$
So inequality (\ref{eq23+}) gives $\gamma(F'_1, \widetilde{F}_1)\geq\delta(F'_1, \widetilde{F}_1)$.

In view of (\ref{eq22+}) that
\[
\gamma(F'_1, A)=\gamma(F'_1, \widetilde{F}_1)-\gamma(A, C)
\]
and
\[
\delta(F'_1, A)=\delta(F'_1, \widetilde{F}_1)-\delta(A, C).
\]
Since $\gamma(A, C)\leq\delta(A, C)$ and $\gamma(F'_1, \widetilde{F}_1)\geq\delta(F'_1, \widetilde{F}_1)$, we get $\gamma(F'_1, A)\geq\delta(F'_1, A)$.
Note that $F'_1\overset{c}{\prec}A, A\overset{c}{\longrightarrow}C \in\mathcal{R}_1(k)$,  where $c=1\in [k_1-k]$, by induction hypothesis, $\gamma(A, C)>\delta(A, C)$. A contradiction to our assumption. This completes the proof of Claim \ref{00000}.
\end{proof}

By (\ref{eq22}) and (\ref{eq22+}), we have $G'_1\overset{c}{\longrightarrow}E, A\overset{c}{\longrightarrow} C,\max G'_1=\max A, \max E=\max C$ in $\mathcal{R}_1(k)$. By Claim \ref{equal} and Claim \ref{00000}, we get
\begin{equation}\label{hh}
\gamma(G'_1, E)>\delta(G'_1, E).
\end{equation}

\begin{claim}\label{claim6.24}
$\gamma(D, H'_1)>\delta(D, H'_1)$.
\end{claim}

\begin{proof}
Since $G'_1<D$, $E<H'_1$  and $G'_1\prec D \prec E \prec H'_1$ in $\mathcal{R}_1(k)$, 
we will meet the following two cases: $\ell(D)=\ell(H'_1)$ and $\ell(D)<\ell(H'_1)$.
We first consider the case $\ell(D)=\ell(H'_1)$. Then
by Claim \ref{0000}, we have $\gamma(G'_1, D)=\gamma(E, H'_1)$ and $\delta(G'_1, D)=\delta(E, H'_1)$. Therefore,
\begin{align*}
\gamma(D, H'_1)&=\gamma(G'_1, E)-\gamma(G'_1, D)+\gamma(E, H'_1)\\
&=\gamma(G'_1, E)\\
&>\delta(G'_1, E)\\
&=\delta(G'_1, E)-\delta(G'_1, D)+\delta(E, H'_1)\\
&=\delta(D, H'_1),
\end{align*}
as desired. Next we assume that $\ell(D)<\ell(H'_1)$. 
If $c=1$, then $G'_1<D=E<H'_1$. By Claim \ref{claim1}, 
\[
\gamma(D, H'_1)\geq \gamma(G'_1, D)\overset{(\ref{hh})}{>}\delta(G'_1, E)=\delta(D, H'_1),
\]
where the last equality holds by Proposition \ref{clm23}, as required.
If $c\geq 2$, then $G'_1<D=\widetilde{G}_1\precneqq E<H'_1$, $\max D=\max E=n-k$ and $\ell(E)>\ell(D)\geq 0$. By Claim \ref{claim1} and Remark \ref{add},
\[
\gamma(D, H'_1)\geq \gamma(G'_1, E)\overset{(\ref{hh})}{>}\delta(G'_1, E)=\delta(D, H'_1),
\]
as required.
\end{proof}
Consequently,
$f(D)<f(H'_1)$ following from Claim \ref{claim6.24}. Recall that $D\overset{1}{\longrightarrow}\widetilde{G}_1$ and $H'_1\overset{1}{\longrightarrow}\widetilde{H}_1$. Hence, $f(\widetilde{G}_1)<f(\widetilde{H}_1)$ by applying Claim \ref{equal}. This implies $\gamma(G_1, H_1)>\delta(G_1, H_1)$, as desired.

This  completes the proof of  Lemma \ref{lemma star}.
\end{proof}

\section{Verify Unimodality: The Proofs of Lemmas \ref{c2.28}, \ref{c2.29} and Proposition \ref{claim3.20}}\label{sec7}
 Lemmas \ref{c2.28} and \ref{c2.29}  will  follow from the following lemma.

\begin{lemma}\label{c2.28+}
Let $B_0=\{b_1, \dots, b_x\} \cup [y, y+k]$ with $b_x<y-1$, $k\geq 1$ and $y+k<n$. For $i\in [k]$, let $B_i=\{b_1, \dots, b_x\} \cup [y, y+k-i]\cup [n-i+1, n]$.
Suppose that  $B_i, B_{i+1}, B_{i+2}\in \mathcal{R}_{1}$ for some $i\in [0, k-2]$ and $f(B_i)\leq f(B_{i+1})$, then $f(B_{i+1})< f(B_{i+2})$.
\end{lemma}

Let us  explain why Lemma \ref{c2.28+} implies Lemmas \ref{c2.28} and \ref{c2.29}.
Let $F_2, G_2, H_2$ be as in Lemma \ref{c2.28} or Lemma \ref{c2.29}.
 Let $F_1, G_1, H_1$ be the $k_1$-parities of $F_2, G_2, H_2$ respectively (Proposition \ref{prop2.9} guarantees that $F_1, G_1, H_1$ exists). 
 
 By Fact \ref{fact2.7}, for each $i\in [3, t]$, $F_1$ and $F_2$ have the same $k_i$-partner. Take $s=2$ (see Definition \ref{def5.1}), we have $g(F_2)=f(F_1)$, $g(G_2)=f(G_1)$ and $g(H_2)=f(H_1)$. We may take $B_i=F_1$, $B_{i+1}=G_1$ and $B_{i+2}=H_1$. So $g(F_2)\leq g(G_2)$ implies
\[
f(B_i)=f(F_1)=g(F_2)\leq g(G_2)=f(G_1)=f(B_{i+1}).
\]
Then by Lemma \ref{c2.28+}, we have $f(B_{i+2})>f(B_{i+1})$. So
\[
g(H_2)=f(H_1)=f(B_{i+2})>f(B_{i+1})=f(G_1)=g(G_2).
\]

\begin{proof}[Proof for Lemma \ref{c2.28+}]
Clearly, $\ell(B_{i+1})=\ell(B_i)+1\geq 1$ and $\ell(B_{i+2})=\ell(B_{i+1})+1$.
Let $B'_i$ and $B'_{i+1}$ be the $k_1$-sets such that $B_i<B'_i$ and $B_{i+1}<B'_{i+1}$. Then $B'_i\overset{\ell(B_{i+1})}{\longrightarrow}B_{i+1}$. 
Let $J$ be the $k_1$-set  such that $B'_{i+1}\overset{\ell(B_{i+1})}{\longrightarrow}J$.
Then $B'_i, B_{i+1}$ are $\ell(B_{i+1})$-sequential and $B'_{i+1}, J$ are $\ell(B_{i+1})$-sequential. Clearly, $\max B'_i=\max B'_{i+1}$ and $\max B_{i+1}=\max J=n$.
 By Claim \ref{equal},
\begin{align}\label{e26}
\text{$\gamma(B'_i, B_{i+1})=\gamma(B'_{i+1}, J)$ and $\delta(B'_i, B_{i+1})=\delta(B'_{i+1}, J)$. }
\end{align}

If $\ell(B'_i)\geq 1$, then $B_i<B'_i=B_{i+1}<B_{i+2}=B'_{i+1}$.
By Proposition \ref{clm23}, $\delta(B_i, B_{i+1})=\delta(B_{i+1}, B_{i+2})$. 
By Claim \ref{claim1}, 
we have $\gamma(B_i, B_{i+1})\leq \gamma(B_{i+1}, B_{i+2})$.
So $f(B_{i+1})-f(B_i)=\gamma(B_i, B_{i+1})-\delta(B_i, B_{i+1})\leq \gamma(B_{i+1}, B_{i+2})-\delta(B_{i+1}, B_{i+2})=f(B_{i+2})-f(B_{i+2})$.
So if $f(B_i)< f(B_{i+1})$, then $f(B_{i+1})< f(B_{i+2})$, as desired.
We next assume that
 $f(B_i)= f(B_{i+1})$. Then $\delta(B_i, B_{i+1})=\gamma(B_i, B_{i+1})\geq s'$.
If  $\gamma(B_i, B_{i+1})=s'$, then $\delta(B_i, B_{i+1})=s'$. Since $B_i<B_{i+1}$ and $\max B_{i+1}=n$, $\delta(B_i, B_{i+1})=\beta(B_i, B_{i+1})=0$ or $1$, and $\delta(B_i, B_{i+1})=\beta(B_i, B_{i+1})=s'$ if and only if $n=k_1+k_t=k_2+k_{t-1}=\dots=k_s'+k_{t-s'+1}$ (see Proposition \ref{clm23}). This is a contradiction to $n>k_1+k_t$ (since $s'<s$ and $n\geq k_s+k_t$). So $\gamma(B_i, B_{i+1})>s'$. Consequently, there exists $j\in [s'+1, s]$ such that $\alpha_j(B_i, B_{i+1})>0$. Let $j$ be any integer that satisfies the above condition. By Claim \ref{claim1}, $\ell(G_{i+1})\geq k_1-k_j$. Since $\ell(B_{i+2})=\ell(B_{i+1})+1$, $\ell(G_{i+1})> k_1-k_j$. 
By Claim \ref{claim1} again, $\alpha_j(B_{i+1}, B_{i+2})>\alpha_j(B_i, B_{i+1})$. 
By the arbitrariness of $j$ and the definitions of $\gamma(B_i, B_{i+1})$ and $\gamma(B_{i+1}, B_{i})$ (see Definition \ref{ab}), we conclude that 
$\gamma(B_{i+1}, B_{i})<\gamma(B_i, B_{i+1})$. Combining with $\delta(B_i, B_{i+1})=\delta(B_{i+1}, B_{i+2})$, we have 
$f(B_{i+1})-f(B_i)=\gamma(B_i, B_{i+1})-\delta(B_i, B_{i+1})< \gamma(B_{i+1}, B_{i+2})-\delta(B_{i+1}, B_{i+2})=f(B_{i+2})-f(B_{i+2})$.
So if $f(B_i)= f(B_{i+1})$, then $f(B_{i+1})< f(B_{i+2})$, as desired.

Next we  assume that $\ell(B'_i)=0$.
By Claim \ref{claim1}, for each $j\in [s'+1, s]$, $\alpha_j(B_i, B'_i)=\alpha_j(B_{i+1}, B'_{i+1})=0$. Clearly, for for each $j\in [s']$, $\alpha_j(B_i, B'_i)=\alpha_j(B_{i+1}, B'_{i+1})=1$ and then $\gamma(B_i, B'_i)=\gamma(B_{i+1}, B'_{i+1})$.
Combining with (\ref{e26}), we get
\begin{align*}
\text{$\gamma(B_i, B_{i+1})=\gamma(B_{i+1}, J)$ and $\delta(B_i, B_{i+1})=\delta(B_{i+1}, J)$. }
\end{align*}
Thus $f(J)\geq f(B_{i+1})$ since $f(B_i)\leq f(B_{i+1})$.
Note that $\ell(B_{i+1})=\ell(J)$ and $B_{i+1}\setminus B_{i+1}^{\rm t}\in \mathcal{R}_1(\ell(B_{i+1}))$, so $J\setminus J^{\rm t}\in \mathcal{R}_1(\ell(B_{i+1}))$ and $B_{i+1}\setminus B_{i+1}^{\rm t} \overset{1}{\prec}J\setminus J^{\rm t}$. Since $\ell(B_{i+2})=\ell(B_{i+1})+1$, $B_{i+2}\setminus [n-\ell(B_{i+1})+1, n]\in \mathcal{R}_1(\ell(B_{i+1}))$. And in $\mathcal{R}_1(\ell(B_{i+1}))$,  we have
\[
B_{i+1}\setminus B_{i+1}^{\rm t} \overset{1}{\prec}J\setminus J^{\rm t}\overset{1}{\longrightarrow}B_{i+2}\setminus [n-\ell(B_{i+1})+1, n].
\]
By Lemma \ref{lemma star} and $f(J)\geq f(B_{i+1})$, we obtain $f(B_{i+2})>f(J)\geq f(B_{i+1})$, as required.
\end{proof}

\begin{proof}[Proof for Proposition \ref{claim3.20}]
The proofs of  equalities (\ref{equation13}) and (\ref{equation14}) are quite similar, we prove the previous one only.

Let $G_2=[2, k_2+1]$ and  $G_1$ be the $k_1$-parity of $G_2$. Since $k_1>k_2$, $\ell(G_1)\geq 1$. So
$G_1=[2, k_2+1]\cup [n-k_1+k_2+1, n]$.
Let $A$ be the $k_1$-parity of $[2, k_2]\cup \{n\}$. So
$A=[2, k_2]\cup [n-k_1+k_2, n]$.
To prove (\ref{equation13}), it is equivalent to  prove that  $f(G_1)<\max \{f(\{1\}), f(A)\}$.
Let $k=k_1-k_2$. Note that $G_1\setminus G_1^{\rm t}$, $\{1, n-k_1+2, \dots, n\}\setminus [n-k_1+k_2+1, n]$ and $A\setminus [n-k_1+k_2+1, n]$ are contained in $\mathcal{R}_1(k)$ and
\[
\{1, n-k_1+2, \dots, n\}\setminus [n-k_1+k_2+1, n]<G_1\setminus G_1^{\rm t}\overset{1}{\longrightarrow}A\setminus [n-k_1+k_2+1, n]
\]
in $\mathcal{R}_1(k)$. Let $C_1=G_1\setminus G_1^{\rm t}\cup \{\max (G_1\setminus G_1^{\rm t})+1 \}\cup [n-\ell(G_1)+2, n]$ (if $\ell(G_1)=2$, then $[n-\ell(G_1)+2, n]=\emptyset$). Since $\ell(G_1)\geq 1$, $|C_1|=k_1$. 
Let $B\in \mathcal{R}_1(k)$ be the set such that $G_1\setminus G_1^{\rm t}<B$ and $B_1=B\cup [n-\ell(G_1)+1, n]$. Clearly, $\{1, n-k_1+2, \dots, n\}\precneqq C_1 \precneqq G_1 \precneqq B_1 \prec A$. We may also assume that $f(G_1)\geq f(C_1)$ since otherwise, $g(\mathcal{F}_{2, 3})>g(G_1)=f(G_1)$, we are done. Based on these, we may apply Lemma \ref{c2.28+} to $C_1, G_1, B_1$. Since $f(G_1)\geq f(C_1)$, Lemma \ref{c2.28+} gives $f(G_1)<f(B_1)$. 
 Consequently,
 by Lemma \ref{lemma star}, $f(A)>f(B)>f(G_1)$, as desired.
\end{proof}

\section{Acknowledgements}
This work was supported by NSFC (Grant No. 11931002).

\end{document}